\date{October 2, 2009}
\documentclass[11pt]{amsart}
\usepackage{latexsym,amsmath,amsfonts,amscd,amssymb}
\usepackage{graphics}
\newtheorem{theorem}{Theorem}[section]
\newtheorem{lemma}[theorem]{Lemma}
\newtheorem{proposition}[theorem]{Proposition}
\newtheorem{corollary}[theorem]{Corollary}

\newtheorem{definition}[theorem]{Definition}

\newtheorem{counterexample}[theorem]{Counter-example}
\newtheorem{example}[theorem]{Example}

\theoremstyle{remark}
\newtheorem{remark}[theorem]{Remark}

\newcommand{\la}{\langle}
\newcommand{\ra}{\rangle}

\newcommand{\bd}{\partial}
\newcommand{\x}{\times}

\newcommand{\Vol}{\operatorname{Vol}}

\newcommand{\cA}{{\mathcal A}}

\newcommand{\cC}{{\mathcal C}}

\newcommand{\cM}{{\mathcal M}}

\newcommand{\cL}{{\mathcal L}}

\newcommand{\cS}{{\mathcal S}}
\newcommand{\cT}{{\mathcal T}}

\newcommand{\cV}{{\mathcal V}}

\newcommand{\PP}{{\mathbb P}}
\newcommand{\QQ}{{\mathbb Q}}
\newcommand{\RR}{{\mathbb R}}
\newcommand{\TT}{{\mathbb T}}

\newcommand{\ZZ}{{\mathbb Z}}

\renewcommand{\a}{\alpha}
\renewcommand{\b}{\beta}
\renewcommand{\d}{\delta}
\newcommand{\g}{\gamma}

\renewcommand{\o}{\omega}

\newcommand{\G}{\Gamma}

\begin{document}

\title{Schwartzman cycles and ergodic solenoids}

\subjclass[2000]{Primary: 37A99. Secondary: 58A25, 57R95, 55N45.} \keywords{Real homology,
Ruelle-Sullivan current, Schwartzman current, solenoid, ergodic
theory.}

\author[V. Mu\~{n}oz]{Vicente Mu\~{n}oz}
\address{Instituto de Ciencias Matem\'aticas
CSIC-UAM-UC3M-UCM, Consejo Superior de Investigaciones Cient\'{\i}ficas,
Serrano 113 bis, 28006 Madrid, Spain}

\address{Facultad de
Matem\'aticas, Universidad Complutense de Madrid, Plaza de Ciencias
3, 28040 Madrid, Spain}

\email{vicente.munoz@imaff.cfmac.csic.es}

\author[R. P\'{e}rez Marco]{Ricardo P\'{e}rez Marco}
\address{CNRS, LAGA UMR 7539, Universit\'e Paris XIII, %}
99, Avenue J.-B. Cl\'ement, 93430-Villetaneuse, France}

%\address{Departamento de Matem\'aticas, Universidad Carlos III de Madrid,
%28911 Legan\'es, Madrid}

\email{ricardo@math.univ-paris13.fr}

\thanks{Partially supported through Spanish MEC grant MTM2007-63582.}
%First author supported through MCyT grant MTM2004-07090-C03-01
%(Spain) and NSF grant DMS-0111298 (US). Second author supported by
%CNRS (UMR 7539) and NSF grant DMS-0202494.}

\maketitle

\begin{abstract}
We extend Schwartzman theory beyond dimension $1$ and provide a unified
treatment of Ruelle-Sullivan and Schwartzman theories via Birkhoff's
ergodic theorem for the class of immersions of
solenoids with a trapping region.
\end{abstract}

%%%%%%%%%%%%%%%%%%%%%%%%%%%%%%%%%%%%%%%%%%%%%%%%%%%%%%%%%%%%%%%%%%
\section{Introduction} \label{sec:introduction}
%%%%%%%%%%%%%%%%%%%%%%%%%%%%%%%%%%%%%%%%%%%%%%%%%%%%%%%%%%%%%%%%%%

This is the second paper of
a series of articles \cite{MPM1,MPM3,MPM4,MPM5} in which we aim to
give a geometric realization of {\em real} homology classes in
smooth manifolds. This paper is devoted to the definition of Schwartzman
homology classes and its relationship with the generalized
currents associated to solenoids defined in \cite{MPM1}.

Let $M$ be a smooth manifold. A closed oriented submanifold $N\subset M$
of dimension $k\geq 0$ determines a homology class in
$H_k(M, \ZZ)$.
This homology class in $H_k(M,\RR)$, as dual of De
Rham cohomology, is explicitly given by integration of the
restriction to $N$ of differential $k$-forms on $M$.
Unfortunately, because of topological reasons dating back to Thom
\cite{Thom1}, not all integer homology classes in
$H_k(M,\ZZ )$ can be realized in such a way. Geometrically, we can
realize any class in $H_k(M, \ZZ)$ by topological $k$-chains. The
real homology $H_k(M,\RR)$ classes are only realized by formal
combinations with real coefficients of $k$-cells. This is not
satisfactory for various reasons. In particular, for diverse
purposes it is important to have an explicit realization, as
geometric as possible, of real homology classes.

The first contribution in this direction came in 1957 from the work
of S.\ Schwartzman \cite{Sc}. Schwartzman showed how, by a limiting
procedure, one-dimensional curves embedded in $M$ can define a real
homology class in $H_1(M,\RR)$. More precisely, he proved that this
happens for almost all curves solutions to a differential equation
admitting an invariant ergodic probability measure. Schwartzman's
idea consists on integrating $1$-forms over
large pieces of the parametrized curve and normalizing this integral
by the length of the parametrization. Under suitable conditions, the
limit exists and defines an element  of the dual of $H^1(M,\RR )$,
i.e. an element of $H_1(M, \RR )$. This procedure is equivalent to
the more geometric one of closing large pieces of the curve by
relatively short closing paths. The closed curve obtained defines an
integer homology class. The normalization by the length of the
parameter range provides a class in $H_1(M,\RR )$. Under suitable
hypothesis, there exists a unique limit in real homology when the
pieces exhaust the parametrized curve, and this limit is independent
of the closing procedure.
In sections \ref{sec:1-schwartzman} and \ref{sec:calibrating}, we shall study
this circle of ideas in great generality. In section \ref{sec:1-schwartzman}
we shall define Schwartzman cycles for parametrized and unparametrized curves
in $M$, and study their properties. In section \ref{sec:calibrating}, we
explore an alternative route to define real homology classes associated
to
curves in $M$ by using the universal covering $\pi:\tilde M\to M$.

%\ref{sec:Schwartzman-cycles} and \ref{sec:k-schwartzman}, we study
%the different aspects of the Schwartzman procedure, that we extend
%it to higher dimension.

It is natural to ask whether it is possible to realize every real
homology class using Schwartzman limits. By the result of \cite{MPM4},
we can realize any real homology class by the generalized current
associated to an immersed oriented uniquely
ergodic solenoid. A {\em solenoid} (see \cite{MPM1}) is an abstract
laminated space endowed with a transversal structure. For these
oriented solenoids we can consider $k$-forms that we can integrate
provided that we are given a transversal measure invariant by the
holonomy group. An immersion of a solenoid $S$ into $M$
is a regular map $f: S\to M$ that is an immersion in each leaf. If
the solenoid $S$ is endowed with a transversal measure $\mu=(\mu_T)$, then
any smooth $k$-form in $M$ can be pulled back to $S$ by $f$ and
integrated. The resulting numerical value only depends on the
cohomology class of the $k$-form. Therefore we have defined a closed
current that we denote by $(f,S_\mu )$ and that
call a generalized current \cite{MPM1}. It defines a homology class
$[f,S_\mu] \in H_k(M,\RR )$. This is reviewed in section \ref{sec:minimal}.

In section \ref{sec:Schwartzman-cycles}, we study the relation between
the generalized current defined by an immersed oriented measured
$1$-solenoid $S_\mu$ and the Schwartzman measure defined by any one of its
leaves. The relationship is best expressed for ergodic and uniquely
ergodic solenoids. In the first case, almost all $\mu_T$-leaves define
Schwartzman classes which represent $[f,S_\mu]$. In the second case,
the property holds for all leaves.

Section \ref{sec:k-schwartzman} is devoted to the generalization of the
Schwartzman theory to higher dimensions. For a complete
$k$-dimensional immersed submanifold $N\subset M$ of a Riemannian manifold,
we define a Schwartzman class by taking large balls, closing them
with small caps, normalizing the homology class thus obtained and
finally taking the limit. This process is only possible when such
capping exist. If $S$ is a $k$-solenoid immersed in $M$,
one would naturally expect that there is some
relation between the generalized currents and the Schwartzman
current (if defined) of the leaves. The main
result is that there is such relation for the class of minimal,
ergodic solenoids with a trapping region (see definition
\ref{def:trapping}). For such solenoids, the holonomy group is generated by a single
map. Then the bridge between generalized currents and Schwartzman
currents of the leaves is provided by Birkhoff's ergodic theorem.
We prove the following:

\begin{theorem}\label{thm:1.3}
Let $S_\mu$ be an oriented and minimal solenoid endowed with an
ergodic transversal measure $\mu$, and possessing a trapping region $W$.
Let $f: S_\mu \to M$ be an
immersion of $S_\mu$ into $M$ such that $f(W)$ is
contained in a ball. Then for $\mu_T$-almost all leaves
$l\subset S_\mu$, the Schwartzman homology class of $f(l)\subset M$
is well defined and coincides with the homology class
$[f,S_\mu]$.
\end{theorem}

We are particularly interested in uniquely ergodic solenoids, with
only one ergodic transversal measure. As is well known, in this
situation we have uniform convergence of Birkhoff's sums, which
implies the stronger result:

\begin{theorem}\label{thm:1.4}
Let $S_\mu$ be a minimal, oriented and uniquely ergodic solenoid
which has a trapping region $W$. Let $f:
S_\mu \to M$ be an immersion of $S_\mu$ into $M$ such that $f(W)$ is 
contained in a ball. Then for all
leaves $l\subset S_\mu$, the Schwartzman homology class of
$f(l)\subset M$ is well defined and coincides with the homology class
$[f,S_\mu]$.
\end{theorem}

\noindent \textbf{Acknowledgements.} \
The authors are grateful to Alberto Candel, Etienne Ghys, Nessim Sibony,
Dennis Sullivan and Jaume Amor\'os
for their comments and interest on this work.
The first author wishes to acknowledge Universidad Complutense de
Madrid and Institute for Advanced Study at Princeton for their
hospitality and for providing excellent working conditions.  The second author
thanks Jean Bourgain and the IAS at Princeton for their hospitality
and facilitating the collaboration of both authors.

%%%%%%%%%%%%%%%%%%%%%%%%%%%%%%%%%%%%%%%%%%%%%%%%%%%%%%%%%%%%%%%%%%
\section{Solenoids and generalized currents} \label{sec:minimal}
%%%%%%%%%%%%%%%%%%%%%%%%%%%%%%%%%%%%%%%%%%%%%%%%%%%%%%%%%%%%%%%%%%

Let us review the main concepts introduced in \cite{MPM1}, and that we shall use later in this paper.

\begin{definition}\label{def:k-solenoid} A
$k$-solenoid, where $k\geq 0$, of class $C^{r,s}$, is a compact Hausdorff space endowed with an atlas
of flow-boxes $\cA=\{ (U_i,\varphi_i)\}$,
 $$
 \varphi_i:U_i\to D^k\x K(U_i)\, ,
 $$
where $D^k$ is the $k$-dimensional open ball, and $K(U_i)\subset \RR^l$ is the transversal
set of the flow-box. The changes of charts $\varphi_{ij}=\varphi_i\circ
\varphi_j^{-1}$ are of the form
 \begin{equation}\label{eqn:change-of-charts}
 \varphi_{ij}(x,y)=(X(x,y), Y(y))\, ,
 \end{equation}
where $X(x,y)$ is of class $C^{r,s}$ and $Y(y)$ is of class $C^s$.
\end{definition}

Let $S$ be a $k$-solenoid, and $U\cong D^k \x K(U)$ be a flow-box for $S$. The sets
$L_y= D^k\x \{y\}$ are called the (local) leaves of the flow-box. A leaf $l\subset S$ of the
solenoid is a connected $k$-dimensional manifold whose intersection with any flow-box
is a collection of local leaves. The solenoid is oriented if the leaves are oriented
(in a transversally continuous way).

A transversal for $S$ is a subset $T$ which is a finite union of transversals of flow-boxes.
Given two local transversals $T_1$ and $T_2$ and
a path contained in a leaf from a point of $T_1$ to a point of $T_2$,
there is a well-defined holonomy map $h:T_1\to T_2$. The holonomy maps form a pseudo-group.

A $k$-solenoid $S$ is minimal if it does not contain a proper sub-solenoid. By \cite[section 2]{MPM1},
minimal solenoids exist. If $S$ is minimal, then any transversal is a global transversal, i.e., it
intersects all leaves. In the special case of an oriented minimal $1$-solenoid, the holonomy
return map associated to a local transversal,
 $$
 R_T:T \to T
 $$
is known as the Poincar\'e return map (see \cite[Section 4]{MPM1}).

\begin{definition} \label{def:transversal-measure}
Let $S$ be a $k$-solenoid. A transversal measure $\mu=(\mu_T)$ for
$S$ associates to any local transversal $T$ a locally finite measure
$\mu_T$ supported on $T$, which are invariant by the holonomy
pseudogroup, i.e. if $h : T_1 \to T_2$ is a holonomy map, then
$h_* \mu_{T_1}= \mu_{T_2}$.
\end{definition}

We denote by $S_\mu$ a $k$-solenoid $S$ endowed with a transversal
measure $\mu=(\mu_T)$. We refer to $S_\mu$ as a measured solenoid.
Observe that for any transversal measure $\mu=(\mu_T)$ the scalar
multiple $c\, \mu=(c \, \mu_T)$, where $c>0$, is also a transversal
measure. Notice that there is no natural scalar normalization of
transversal measures.

\begin{definition} \label{def:transversal-ergodicity}\textbf{\em (Transverse ergodicity)}
A transversal measure $\mu=(\mu_T )$ on a solenoid $S$ is ergodic if for any Borel set
$A\subset T$ invariant by the pseudo-group of holonomy maps on $T$,
we have
 $$
 \mu_T(A) = 0 \ \ {\hbox{\rm{ or }}} \ \ \mu_T(A) = \mu_T(T) \, .
 $$
We say that $S_\mu$ is an ergodic solenoid.
\end{definition}

\begin{definition} \label{def:transversal-unique-ergodicity}
Let $S$ be a $k$-solenoid. The solenoid $S$ is
uniquely ergodic if it has a unique (up to scalars)
transversal measure $\mu$ and its support is the whole of $S$.
\end{definition}

\bigskip

Now let $M$ be a smooth manifold of dimension $n$. An immersion of a $k$-solenoid
$S$ into $M$, with $k<n$, is a smooth map $f:S\to M$ such that the differential
restricted to the tangent spaces of leaves has rank $k$ at every
point of $S$. The solenoid $f:S\to M$ is
transversally immersed if for any flow-box $U\subset S$ and chart
$V\subset M$, the map $f:U= D^k\x K(U) \to V\subset \RR^n$ is
an embedding, and the images of
the leaves intersect transversally in $M$. If moreover $f$ is injective, then
we say that the solenoid is embedded.

Note that under a transversal immersion, resp.\ an embedding,
$f:S\to M$, the images of the leaves are immersed, resp.\
injectively immersed, submanifolds.

\bigskip

Let $\cC_k(M)$ denote the space of $k$-dimensional currents on $M$.

\begin{definition}\label{def:Ruelle-Sullivan}
Let $S_\mu$ be an oriented measured $k$-solenoid. An immersion
$f:S\to M$
defines a generalized Ruelle-Sullivan current $(f,S_\mu)\in \cC_k(M)$ as follows.
Let $S=\bigcup_i S_i$ be a measurable partition such that each
$S_i$ is contained in a flow-box $U_i$. For $\omega\in \Omega^k(M)$, we define
 $$
 \la (f,S_\mu),\omega \ra=\sum_i \int_{K(U_i)} \left ( \int_{L_y\cap S_i}
 f^* \omega \right ) \ d\mu_{K(U_i)} (y) \, ,
 $$
where $L_y$ denotes the horizontal disk of the flow-box.
\end{definition}

In \cite{MPM1} it is proved that $(f,S_\mu)$ is a closed current. Therefore, it defines
a real homology class
 $$
  [f,S_\mu]\in H_k(M,\RR)\, .
  $$
In their original article \cite{RS}, Ruelle and Sullivan defined
this notion for the restricted class of solenoids embedded in $M$.

%%%%%%%%%%%%%%%%%%%%%%%%%%%%%%%%%%%%%%%%%%%%%%%%%%%%%%%%%%%%%%%%%%
\section{Schwartzman measures} \label{sec:schwartzman}
%%%%%%%%%%%%%%%%%%%%%%%%%%%%%%%%%%%%%%%%%%%%%%%%%%%%%%%%%%%%%%%%%%

Let $S$ be a Riemannian $k$-solenoid, that is, a solenoid endowed with a Riemmanian
metric on each leaf. In some situations, we may define transversal measures associated
to $S$ by considering large chunks of a single leaf $l\subset S$. These will
be called Schwartzman measures. We start by recalling some notions from \cite[Section 6]{MPM1}.

\begin{definition}\label{def:desintegrate}\textbf{\em (daval
measures)} Let $\mu$ be a measure supported on $S$. The measure
$\mu$ is a daval measure if it desintegrates as volume along leaves
of $S$, i.e.\ for any flow-box $(U,\varphi)$ with local transversal
$T=\varphi^{-1}(\{0\}\x K(U))$, we have a measure $\mu_{U,T}$
supported on $T$ such that for any Borel set $A\subset U$
 \begin{equation}\label{eqn:1111}
 \mu(A)=\int_{T} {\Vol}_k(A_y) \ d\mu_{U,T}(y) \, ,
 \end{equation}
where $A_y=A\cap \varphi^{-1} (D^k\times \{ y \} )\subset U$.
\end{definition}

We denote by $\cM_\cL (S)$ the space of probability
daval measures, by $\cM_\cT(S)$ the space of (non-zero) transversal measures on $S$, and by
$\overline{\cM}_\cT(S)$ the quotient of ${\cM}_\cT(S)$  by
positive scalars.
The following result is Theorem 6.8 in \cite{MPM1}.

\begin{theorem} \label{thm:transverse-riemannian}
\textbf{\em (Tranverse measures of the Riemannian solenoid)} There
is a one-to-one correspondence between transversal measures $(\mu_T)$
and finite daval measures $\mu$. Furthermore, there is an
isomorphism
  $$
  \overline{\cM}_\cT(S) \cong \cM_\cL(S)\, .
  $$
\end{theorem}

The correspondence follows from equation (\ref{eqn:1111}).
If $S$ is a uniquely ergodic Riemannian solenoid, then the above result allows to normalize
the transversal measure in a unique way, by imposing that the corresponding daval measure
has total mass $1$.

Now we introduce a subclass of solenoids for which daval measures do exist.

\begin{definition}\label{def:controlled-growth}\textbf{\em (Controlled growth solenoids)}
Let $S$ be a Riemannian solenoid. Fix a leaf $l\subset S$ and an
exhaustion $(C_n)$ by subsets of $l$. For a flow-box $(U,\varphi)$
write
 $$
 C_n \cap U = A_n\cup B_n \, ,
 $$
where $A_n$ is composed by all full disks $L_y=\varphi^{-1}
(D^k\times \{y\})$ contained in $C_n$, and $B_n$ contains those
connected components $B$ of $C_n\cap U$ such that $B\not= L_{y}\cap
U$ for any $y$. The solenoid $S$ has controlled growth with respect
to $l$ and $(C_n)$ if for any flow-box $U$ in a finite covering of
$S$
 $$
 \lim_{n\to +\infty } \frac{{\Vol}_k(B_n) }{{\Vol}_k(A_n)}=0\, .
 $$

 The solenoid $S$ has controlled growth if $S$ contains a leaf $l$ and an
exhaustion $(C_n)$ such that $S$ has controlled growth with respect
to $l$ and $(C_n)$.
\end{definition}

For a Riemannian solenoid $S$, it is natural to consider the
exhaustion by Riemannian balls $B(x_0,R_n)$ in a leaf $l$ centered at a point
$x_0\in l$ and with $R_n\to +\infty$, and test the controlled growth
condition with respect to such exhaustions.

The controlled growth condition depends a priori on the Riemannian
metric. As we see next, it guarantees the existence of daval
measures, hence the existence of transversal measures on $S$. Indeed
the measures we construct are Schwartzman measures defined as:

\begin{definition}\label{def:Schwartzman limit}\textbf{\em (Schwartzman
limits and measures)} We say that a measure $\mu$ is a Schwartzman
measure if it is obtained as the limit
 $$
 \mu=\lim_{n\to +\infty } \mu_{n}\, ,
 $$
where the measures $(\mu_n)$ are the normalized $k$-volume of the
exhaustion $(C_n)$ (that is, $\mu_n$ are normalized to have total
mass $1$). We denote by
$\cM_\cS (S)$ the space of (probability) Schwartzman measures.
\end{definition}

Compactness of probability measures show:

\begin{proposition}
There are always Schwartzman measures on $S$,
 $$
 \cM_\cS(S)\not= \emptyset \, .
 $$
\end{proposition}

\begin{theorem}\label{thm:desintegration-Schwartzman-measures}
If $S$ is a solenoid with controlled growth, then any Schwartzman
measure is a daval measure,
 $$
 \cM_\cS (S)\subset \cM_\cL (S) \, .
 $$
In particular, $\cM_\cL (S)\not= \emptyset$ and $S$ admits
transversal measures.
\end{theorem}

\begin{proof}
Let $\mu_{n}\to \mu$ be a Schwartzman limit as in definition
\ref{def:Schwartzman limit}. For any flow-box $U$ we prove that
$\mu$ desintegrates as volume on leaves of $U$. Since $S$ has
controlled growth, pick a leaf and an exhaustion which satisfy the
controlled growth condition. Let
 $$
 C_n\cap U = A_n\cup B_n \, ,
 $$
be the decomposition for $C_n\cap U$ described before. The set $A_n$
is composed of a finite number of horizontal disks. We define a new
measure $\nu_n$ with support in $U$ which is the restriction of
$\mu_{n}$ to $A_n$, i.e. it is proportional to the $k$-volume on
horizontal disks. The measure $\nu_n$ desintegrates as volume on
leaves in $U$. The transversal measure is a finite sum of Dirac
measures. Moreover the controlled growth condition implies that
$(\nu_n)$ and $(\mu_{n|U})$ must converge to the same limit. But we
know that $\cM_\cL(S)$ is closed, thus the limit measure $\mu_{|U}$
desintegrates on leaves in $U$. So $\mu$ is a daval measure.
\end{proof}

For uniquely ergodic solenoids we have:

\begin{corollary}\label{cor:volume}
The volume $\mu$ of a uniquely ergodic solenoid with controlled
growth is the unique Schwartzman measure. Therefore there is only
one Schwartzman limit
 $$
 \mu=\lim_{n\to +\infty} \mu_n \, ,
 $$
which is independent of the leaf and the exhaustion.
\end{corollary}

\begin{proof} There are always Schwartzman limits.
Theorem \ref{thm:desintegration-Schwartzman-measures} shows that any
such limit $\mu$ desintegrates as volume on leaves. Thus the measure
$\mu$ defines the unique (up to scalars) transversal measure
$(\mu_T)$. But, conversely, the transversal measure determines the
measure $\mu$ uniquely. Therefore there is only possible limit
$\mu$, which is the volume of the uniquely ergodic solenoid.
\end{proof}

%%%%%%%%%%%%%%%%%%%%%%%%%%%%%%%%%%%%%%%%%%%%%%%%%%%%%%%%%%%%%%%%%%
\section{Schwartzman clusters and asymptotic cycles} \label{sec:1-schwartzman}
%%%%%%%%%%%%%%%%%%%%%%%%%%%%%%%%%%%%%%%%%%%%%%%%%%%%%%%%%%%%%%%%%%

Let $M$ be a compact  $C^\infty$ Riemannian manifold. Observe that
since $H_1(M,\RR)$ is a finite dimensional real vector space, it
comes equipped with a unique topological vector space structure.

The map $\gamma \mapsto [\gamma ]$ that associates to each loop its
homology class in $H_1(M,\ZZ) \subset H_1(M,\RR)$ is continuous when
the space of loops is endowed with the Hausdorff topology.
Therefore, by compactness, oriented rectifiable loops in $M$ of
uniformly bounded length define a bounded set in $H_1(M,\RR)$.

We have a more precise quantitative version of this result.

\begin{lemma} \label{lem:9.1}
Let $(\g_n)$ be a sequence of oriented rectifiable loops in $M$, and
$(t_n)$ be a sequence with $t_n>0$ and $t_n\to +\infty$. If
 $$
 \lim_{n\to +\infty } \frac{l(\g_n )}{t_n}=0 \, ,
 $$
then in $H_1(M,\RR)$ we have
 $$
 \lim_{n\to +\infty} \frac{[\g_n]}{t_n} =0 \, .
 $$
\end{lemma}

\begin{proof}
Via the map
 $$
 \omega \mapsto \int_\g \omega \, ,
 $$
each loop $\g$ defines a linear map $L_\g$ on $H^1(M,\RR )$ that
only depends on the homology class of $\g$. We can extend this map
to $\RR \otimes H_1(M,\ZZ)$ by
 $$
 c\otimes \g \mapsto c \cdot L_\g \, .
 $$

We have the isomorphism
 $$
 H_1(M,\RR)=\RR \otimes H_1(M,\ZZ) \cong \left ( H^1(M,\RR ) \right )^*
 \, .
 $$
The Riemannian metric gives a $C^0$-norm on forms. We consider the
norm in $H^1(M,\RR )$ given as
 $$
 || [\omega]||_{C^0} =\min_{\omega \in [\omega]} ||\omega|| \, ,
 $$
and the associated operator norm in $H_1(M,\RR )\cong \left (
H^1(M,\RR ) \right )^*$.

We have
 $$
 |L_\gamma([\omega])|= \left | \int_\g \omega \right |\leq l(\g ) ||\o ||_{C^0}
 \leq l(\g ) ||[\o ] ||_{C^0} \, ,
 $$
so
 $$
 ||L_\g || \leq l(\g ) \, .
 $$
Hence $l(\g_{n})/t_n\to 0$ implies $L_{\g_{n}}/t_n \to 0$ which is
equivalent to $[\g_n]/t_n\to 0$.
\end{proof}

\begin{definition}\label{def:asymptotic-cycle} \textbf{\em
(Schwartzman asymptotic $1$-cycles)} \label{def:9.2} Let $c$ be a
parametrized continuous curve $c:{\RR}\to M$ defining an immersion
of $\RR$. For $s, t\in \RR$, $s<t$, we choose a rectifiable
oriented curve $\gamma_{s,t}$ joining $c(s)$ to $c(t)$ such that
 $$
 \lim_{t\to +\infty \atop s\to -\infty}  \frac{l(\g_{s, t} )}{t-s}=0 \
 .
 $$

The parametrized curve $c$ is a Schwartzman asymptotic $1$-cycle
if the juxtaposition of $c|_{[s,t]}$ and $\gamma_{s,t}$, denoted
$c_{s ,t}$ (which is a $1$-cycle), defines a homology class
$[c_{s,t}]\in H_1(M,{\ZZ})$ such that the limit
 \begin{equation}\label{eqn:X}
 \lim_{t\to +\infty \atop s \to -\infty} \frac{[c_{s,t}]}{t-s} \in H_1(M,{\RR})
 \end{equation}
exists.

We define the Schwartzman asymptotic homology class as
 $$
 [c]:=\lim_{t\to +\infty \atop s \to -\infty} \frac{[c_{s,t}]}{t-s} \, .
 $$
\end{definition}

Thanks to lemma \ref{lem:9.1} this definition does not depend on
the choice of the closing curves $(\g_{s,t})$. If we take another
choice $(\g'_{s,t})$, then as homology classes,
 $$
 [c_{s,t}]=[c'_{s,t}]+[\g'_{s,t}-\g_{s,t}]\, ,
 $$
and
 $$
 \frac{l(\g_{s,t}'-\g_{s,t})}{t-s}=
 \frac{l(\g_{s,t}')}{t-s}+\frac{l(\g_{s,t})}{t-s}\to 0\, ,
 $$
as $t\to \infty$, $s\to -\infty$. By lemma \ref{lem:9.1},
 $$
 \lim_{t\to +\infty \atop s \to -\infty}
 \frac{[\g_{s,t}-\g'_{s,t}]}{t-s}=0 \, ,
 $$
thus
 $$
 [c]=\lim_{t\to +\infty \atop s \to -\infty}
 \frac{[c_{s,t}]}{t-s}=\lim_{t\to +\infty \atop s \to -\infty}
 \frac{[c'_{s,t}]}{t-s} \, .
 $$

Note that we do not assume that $c(\RR )$ is an embedding of
$\RR$, i.e. $c(\RR )$ could be a loop. In that case, the
Schwartzman asymptotic homology class coincides with a scalar
multiple (the scalar depending on the parametrization) of the
integer homology class $[c(\RR)]$. This shows that the Schwartzman
homology class is a generalization to the case of immersions
$c:\RR\to M$. More precisely we have:

\begin{proposition} \label{prop:9.3}
If $c: \RR \to M$ is a loop then it is a Schwartzman asymptotic
$1$-cycle and the Schwartzman asymptotic homology class is a
scalar multiple of the homology class of the loop  $[c(\RR )]\in
H_1(M,\ZZ)$.

If $c: \RR \to M$ is a rectifiable loop with its arc-length
parametrization, and $l(c)$ is the length of the loop $c$, then
 $$
 [c]=\frac{1}{l(c)}\, [c(\RR)] \, .
 $$
\end{proposition}

\begin{proof}
Let $t_0>0$ be the minimal period of the map $c:\RR \to M$. Then
 $$
 [c_{s,t}]= \left [\frac{t-s}{t_0} \right ] [c(\RR )] + O (1) \, .
 $$
Then
 $$
 \lim_{t\to +\infty \atop s\to -\infty} \frac{[c_{s,t}]}{t-s}
 =\frac{1}{t_0} [c(\RR )] \, .
 $$
When  $c: \RR \to M$ is the arc-length parametrization of a
rectifiable loop, the period $t_0$ coincides with the length of
the loop.
\end{proof}

We will assume also in the definition of Schwartzman asymptotic
$1$-cycle that we choose $(\g_{s,t})$ such that
$l(\g_{s,t})/(t-s)\to 0$ uniformly and separately on $s$ and $t$
when $t\to +\infty$ and $s\to -\infty$. For simplicity we can
decide to choose always $\g_{s,t}$ with uniformly bounded length,
and even with $\{\g_{s,t} ; s<t\}$ contained in a compact subset
of the space of continua of $M$. Then the uniform boundedness will
hold for any Riemannian metric and the notions defined will not
depend on the Riemannian structure.

\begin{definition}\textbf{\em (Positive and negative asymptotic cycles)}
Under the assumptions of definition \ref{def:9.2}, if the limit
 \begin{equation}\label{eqn:Y}
 \lim_{t\to +\infty} \frac{[c_{s,t}]}{t-s} \in H_1(M,{\RR})
 \end{equation}
exists then it does not depend on $s$, and we say that the
parametrized curve $c$ defines a positive asympotic cycle. The
positive Schwartzman homology class is defined as
 $$
 [c_+]=\lim_{t\to +\infty} \frac{[c_{s,t}]}{t-s} \, .
 $$

The definition of negative asymptotic cycle and negative
Schwartzman homology class is the same but taking $s\to -\infty$,
 $$
 [c_-]=\lim_{s\to -\infty} \frac{[c_{s,t}]}{t-s} \, .
 $$
\end{definition}

The independence of the limit (\ref{eqn:Y}) on $s$ follows from
 $$
  \lim_{t\to +\infty} \frac{[c_{s',t}]}{t-s'} =
  \lim_{t\to +\infty} \frac{[c_{s,t}] + [c_{s',s}] +O(1)}{t-s} \cdot \frac{t-s}{t-s'} =
  \lim_{t\to +\infty} \frac{[c_{s,t}]}{t-s}
 \, .
 $$

\begin{proposition}\label{prop:criterium-Schwartzman-cycle}
A parametrized curve $c$ is a Schwartzman asymptotic $1$-cycle if
and only if it is both a positive and a negative asymptotic cycle
and
 $$
 [c_+]=[c_-] \, .
 $$
In that case we have
 $$
 [c]=[c_+]=[c_-] \, .
 $$
\end{proposition}

\begin{proof}
If $c$ is a Schwartzman asymptotic $1$-cycle, then for $t\to +\infty$ take
$s\to -\infty$ very slowly, say satisfying the relation $t=s^2 \, l(c_{|[s,0]})$,
which defines $s=s(t)<0$ uniquely as a function of $t>0$. Then
 $$
 \begin{aligned} \
 [c]= & \lim_{t\to \infty \atop s =s(t) \to -\infty} \frac{[c_{s,t}]}{t-s} =
  \lim_{t\to +\infty} \frac{[c_{s,0}] +[c_{0,t}]+O(1)}{t-s} \\
  =&
  \lim_{t\to +\infty} \left(\frac{[c_{s,0}]+O(1)}{t}
   + \frac{[c_{0,t}]}{t} \right)\frac{t}{t-s} = \lim_{t\to +\infty}
 \frac{[c_{0,t}]}{t} \, ,
 \end{aligned}
 $$
since $\frac{t}{t-s} \to 1$ because $\frac{s}{t}\to 0$, and $\frac{[c_{s,0}]}{t}\to
0$ by lemma \ref{lem:9.1}.
So $c$ is a positive asymptotic cycle and $[c]=[c_+]$.
Analogously, $c$ is a negative asymptotic cycle and $[c]=[c_-]$.

Conversely, assume that $c$ is a positive and negative asymptotic
cycle with $[c_+]=[c_-]$. For $t$ large we have
 $$
 \frac{[c_{0,t}]}{ t}=[c_+]+o(1) \, .
 $$
For $-s$ large we have
 $$
 \frac{[c_{s,0}]}{-s}=[c_-]+o(1) \, .
 $$
Now
  $$
  \frac{[c_{s,t}]}{t-s}= \frac{-s}{t-s} \cdot\frac{[c_{s,0}]}{-s} +\frac{t}{t-s} \cdot\frac{[c_{0,t}]}{t} +
  \frac{O(1)}{t-s} = \frac{-s}{t-s} [c_+] + \frac{t}{t-s} [c_-]
  +o(1)\, .
  $$
As $[c_+]=[c_-]$, we get that this limit exists and equals
$[c]=[c_+]=[c_-]$.
\end{proof}

\begin{definition}\textbf{\em (Schwartzman clusters)}
Under the assumptions of definition \ref{def:9.2}, we can
consider, regardless of whether (\ref{eqn:X}) exists or not, all
possible limits
 \begin{equation}\label{eqn:XX}
 \lim_{n\to +\infty} \frac{[c_{s_n,t_n}]}{t_n-s_n} \in H_1(M,\RR ) \, ,
 \end{equation}
with $t_n\to +\infty$ and $s_n\to -\infty$, that is, the derived
set of $([c_{s,t}]/(t-s))_{t\to \infty,s\to -\infty}$. The limits
(\ref{eqn:XX}) are called Schwartzman asymptotic homology classes
of $c$, and they form the Schwartzman cluster of $c$,
 $$
 \cC (c)\subset H_1(M,\RR) \, .
 $$
A Schwartzman asymptotic homology class (\ref{eqn:XX}) is balanced
when the two limits
 $$
 \lim_{n\to +\infty} \frac{[c_{0,t_n}]}{t_n} \in H_1(M,\RR ) \, ,
 $$
and
 $$
 \lim_{n\to +\infty} \frac{[c_{s_n,0}]}{ -s_n} \in H_1(M,\RR ) \, ,
 $$
do exist in $H_1(M,\RR)$. We denote by $\cC_b (c)\subset \cC
(c)\subset H_1(M,\RR)$ the set of those balanced Schwartzman
asymptotic homology classes. The set $\cC_b (c)$ is named the
balanced Schwartzman cluster.

We define also the positive and negative Schwartzman clusters,
$\cC_+ (c)$ and $\cC_- (c)$, by taking only limits $t_n\to
+\infty$ and $s_n\to -\infty$ respectively.
\end{definition}

\begin{proposition}\label{prop:closed-clusters}
The Schwartzman clusters $\cC (c)$, $\cC_+ (c)$ and $\cC_- (c)$
are closed  subsets of $H_1(M,\RR)$.

If $\{[c_{s,t}]/(t-s) ; s<t\}$ is bounded in $H_1(M,\RR )$, then
the Schwartzman clusters $\cC (c)$, $\cC_+ (c)$ and $\cC_- (c)$
are non-empty, compact and connected  subsets of $H_1(M,\RR)$.
\end{proposition}

\begin{proof}
The Schwartzman cluster $\cC (c)$ is the derived set of
 $$
 ( [c_{s,t}]/(t-s) )_{t\to \infty, s\to - \infty} \, ,
 $$
in $H_1(M,\RR)$, hence closed.

Under the boundedness assumption, non-emptiness and compactness
follow. Also the oscillation of $([c_{s,t}])_{s,t}$ is bounded by
the size of $[\g_{s,t}]$. Therefore the magnitude of the
oscillation of $([c_{s,t}]/(t-s))_{s,t}$ tends to $0$ as $t\to
\infty$, $s\to -\infty$. This forces the derived set to be
connected under the boundedness assumption, since it is
$\epsilon$-connected for each $\epsilon >0$. (A compact metric
space is $\epsilon$-connected for all $\epsilon >0$ if and only if
it is connected.)

Also $\cC_+(c)$, resp.\ $\cC_-(c)$, is closed because it is the
derived set of
 $$
 ( [c_{0,t}]/t)_{t\to \infty} \, ,
 $$
resp.\  $$
 ( [c_{s,0}]/(-s))_{s\to - \infty}\, ,
 $$
in $H_1(M,\RR)$. Non-emptiness, compactness and connectedness
under the boundedness assumption follow for the cluster sets
$\cC_\pm (c)$ in the same way as for $\cC (c)$.
\end{proof}

Note that all these cluster sets may be empty if the parametrization
is too fast.

The balanced Schwartzman cluster $\cC_b (c)$ does not need to be
closed, as shown in the following counter-example.

\begin{counterexample} \label{counter}
We consider the torus $M=\TT^2$. We identify $H_1(M,\RR) \cong
\RR^2$, with $H_1(M,\ZZ )$ corresponding to the lattice $\ZZ^2
\subset \RR^2$. Consider a line $l$ in $H_1(M, \RR^2)$ of
irrational slope passing through the origin, $y=\sqrt{2}\  x$ for
example. We can find a sequence of pairs of points $(a_n,b_n)\in
\ZZ^2 \times \ZZ^2$ in the open lower half plane $H_l$ determined
by the line $l$, such that the sequence of segments $[a_n, b_n]$
do converge to the line $l$, and the middle point $(a_n+b_n)/2 \to
0$ (this is an easy exercise in diophantine approximation). We
assume that the first coordinate of $b_n$ tends to $+\infty$, and
the first coordinate of $a_n$ tends to $-\infty$. Now we can
construct a parametrized curve $c$ on $\TT^2$ such that for all
$n\geq 1$ there are an infinite number of times $t_{n,i}\to
+\infty$ with $[c_{0,t_{n,i}}]/t_{n,i}=b_n$, and for an infinite
number of times $s_{n,i}\to -\infty$,
$[c_{s_{n,i},0}]/(-s_{n,i})=a_n$. Thus in homology the curve $c$
oscillates wildly.  We can adjust the velocity of the
parametrization so that $-s_{n,i}=t_{n,i}$. Hence for these times
 $$
 \frac{[c_{s_{n,i},t_{n,i}}]}{t_{n,i}-s_{n,i}} =
 \frac{a_n (-s_{n,i}) + b_n(t_{n,i}) +O(1)}{t_{n,i}-s_{n,i}}
 \to \frac{a_n+b_n}2\, ,
 $$
when $i\to +\infty$, and the two ends balance each other. We have
great freedom in constructing $c$, so that we may arrange to have
always $[c_{s,t}]\subset H_l$. Then we get that $0\in \cC (c)$ and
all $(a_n+b_n)/2\in \cC_b (c)$ but $0\notin \cC_b (c)$.
\end{counterexample}

\medskip

We have that  $c$ is a Schwartzman asymptotic $1$-cycle (resp.\
positive, negative) if and only if $\cC (c)$ (resp.\ $\cC_+(c)$,
$\cC_-(c)$) is reduced to one point. In that case the Schwartzman
asymptotic $1$-cycle is balanced. The next result generalizes
proposition \ref{prop:criterium-Schwartzman-cycle}. We need first
a definition.

\begin{definition}
Let $A,B\subset V$ be subsets of a real vector space $V$. For
$a,b\in V$ the segment $[a,b]\subset V$ is the convex hull of $\{ a,
b\}$ in $V$. The additive hull of $A$ and $B$ is
 $$
 A\widehat + B=\bigcup_{a\in A \atop b\in B } [a,b] \, .
 $$
\end{definition}

\begin{proposition}\label{prop:balanced-cluster}
The Schwartzman balanced cluster $\cC_b (c)$ is contained in the
additive hull of $\cC_+(c)$ and $\cC_-(c)$
 $$
 \cC_b (c)\subset \cC_+(c)\widehat + \, \cC_-(c)\, .
 $$
Moreover, for each $a\in \cC_+(c)$ and $b\in \cC_-(c)$, we have
 $$
 \cC_b (c) \cap [a,b] \not= \emptyset \, .
 $$
\end{proposition}

\begin{proof}
Let $x\in \cC_b (c)$,
 $$
 x=\lim_{n\to +\infty  } \frac{[c_{s_n,t_n}]}{t_n-s_n} \, .
 $$
We write
 $$
 \frac{[c_{s_n,t_n}]}{t_n-s_n}=\frac{[c_{s_n,0}]}
 {-s_n}\cdot\frac{-s_n}{t_n-s_n}
 +\frac{[c_{0,t_n}]}{t_n}\cdot\frac{t_n}{t_n-s_n} +o(1) \, ,
 $$
and the first statement follows.

For the second, consider
 $$
 a=\lim_{n\to +\infty} \frac{[c_{0,t_n}]}{t_n} \in \cC_+(c) \, ,
 $$
and
 $$
 b=\lim_{n\to +\infty} \frac{[c_{s_n,0}]}{-s_n} \in \cC_-(c) \, .
 $$
Then taking any accumulation point $\tau\in [0,1]$ of the sequence
$(t_n/(t_n-s_n))_n \subset [0,1]$ and taking subsequences in the
above formulas, we get a balanced Schwartzman homology class
 $$
 c=\tau a+(1-\tau )b \in \cC_b(c) \, .
 $$
\end{proof}

\begin{corollary} \label{cor:8.11}
If $\cC_+(c)$ and $\cC_-(c)$ are non-empty, then $\cC_b(c)$ is
non-empty, and therefore $\cC (c)$ is also non-empty.
\end{corollary}

Note that we can have $\cC_+(c)=\cC_-(c)=\emptyset$ (then $\cC_b
(c)=\emptyset$) but $\cC (c)\not= \emptyset$ (modify appropriately
counter-example \ref{counter}).

\medskip

There is one situation where we can assert that the balanced
Schwartzman cluster set is closed.

\begin{proposition} \label{prop:9.11}
If $B=\{[c_{s,t}]/(t-s); s<t\} \subset H_1(M,\RR )$ is a bounded
set, then $\cC (c)$, $\cC_+(c)$, $\cC_-(c)$ and $\cC_b (c)$ are
all compact sets. More precisely, they are all contained in the
convex hull of ${\overline B}$.
\end{proposition}

\begin{proof}
Obviously $\cC (c)$, $\cC_+(c)$ and $\cC_-(c)$ are bounded as
cluster sets of bounded sets, hence compact by proposition
\ref{prop:closed-clusters}.

In order to prove that $\cC_b (c)$ is bounded, we observe that the
additive hull of bounded sets is bounded, therefore boundedness
follows from proposition \ref{prop:balanced-cluster}. We show that
$\cC_b (c)$ is closed. Since $\cC_b(c)\subset \cC (c)$ and
$\cC(c)$ is closed, any accumulation point $x$ of $\cC_b(c)$ is in
$\cC (c)$. Let
 $$
 x=\lim_{n\to +\infty } \frac{[c_{s_n,t_n}]}{t_n-s_n} \, ,
 $$
and write as before
 $$
 \frac{[c_{s_n,t_n}] }{t_n-s_n}=
 \frac{[c_{s_n,0}]}{-s_n}\cdot\frac{-s_n}{t_n-s_n} +\frac{[c_{0,t_n}] }{t_n}
 \cdot\frac{t_n}{t_n-s_n} +o(1) \, .
 $$
Note that $([c_{s_n,0}] /(-s_n) )_n$ and $([c_{0,t_n}] /t_n )_n$
stay bounded. Therefore we can extract converging subsequences and
also for the sequence $(t_n/(t_n-s_n))_n\subset [0,1]$. The limit
along these subsequences $t_{n_k}\to +\infty$ and $s_{n_k}\to
-\infty$ give the same Schwartzman homology class $x$ which turns
out to be balanced.

The final statement follows from the above proofs.
\end{proof}

The situation described in proposition \ref{prop:9.11} is indeed
quite natural. It arises each time that $M$ is a Riemannian manifold
and $c$ is an arc-length parametrization of a rectifiable curve. In
the following proposition we make use of the natural norm
$||\cdot||$ in the homology of a Riemannian manifold defined in the
\ref{sec:appendix1-norm}.

\begin{proposition}\label{prop:Riemannian-clusters}
Let $M$ be a Riemannian manifold and denote by $||\cdot||$ the
norm in homology. If $c$ is a rectifiable curve parametrized by
arc-length then the cluster sets $\cC (c)$, $\cC_+(c)$, $\cC_-(c)$
and $\cC_b (c)$ are compact subsets of $\bar B(0,1)$, the closed
ball of radius $1$ for the norm in homology.

So $\cC (c)$ and $\cC_\pm(c)$ are non-empty, compact and connected,
and $\cC_b (c)$ is non-empty and compact.
\end{proposition}

\begin{proof}
Observe that we have
 $$
 l(c_{s,t})=l(c|_{[s,t]})+l(\g_{s,t})=t-s+l(\g_{s,t}) \, .
 $$
Thus
 $$
 l([c_{s,t}])\leq t-s+ l(\g_{s,t}) \, .
 $$
By theorem \ref{thm:A.4},
 $$
 || [c_{s,t}]||\leq t-s + l(\g_{s,t}) \, ,
 $$
and
 $$
 \left\| \frac{[c_{s,t}]}{t-s} \right\| \leq 1+\frac{l(\g_{s,t} )}{t-s} \, .
 $$
Since $\frac{l(\g_{s,t} )}{t-s}\to 0$ uniformly, we get that
$B=\{[c_{s,t}]/(t-s); s<t\} \subset H_1(M,\RR )$ is a bounded set.

By proposition \ref{prop:closed-clusters}, $\cC (c)$ and
$\cC_\pm(c)$ are non-empty, compact and connected.
By corollary \ref{cor:8.11}, $\cC_b (c)$ is non-empty and by
proposition \ref{prop:9.11}, it is compact.
\end{proof}

Obviously the previous notions depend heavily on the
parametrization. For a non-parametrized curve we can also define
Schwartzman cluster sets.

\begin{definition}\label{def:8.14.si}
For a non-parametrized oriented curve $c\subset M$, we define the
Schwartzman cluster $\cC (c)$ as the union of the Schwartzman
clusters for all orientation preserving parametrizations of $c$.
We define the positive $\cC_+(c)$, resp.\ negative $\cC_-(c)$,
Schwartzman cluster set as the union of all positive, resp.\
negative, Schwartzman cluster sets for all orientation preserving
parametrizations.
\end{definition}

\begin{proposition} \label{prop:8.15.si}
For an oriented curve $c \subset M$ the Schwartzman clusters $\cC
(c)$, $\cC_+(c)$ and $\cC_-(c)$ are non-empty closed cones of
$H_1(M,\RR)$. These cones are degenerate (i.e. reduced to $\{
0\}$) if and only if $\{ [c_{s,t}] ; s<t \}$ is a bounded subset
of $H_1(M, \ZZ )$.
\end{proposition}

\begin{proof}
We can choose the closing curves $\g_{s,t}$ only depending on
${c}(s)$ and ${c}(t)$ and not on the parameter values $s$ and $t$,
nor on the parametrization. Then the integer homology class
$[{c}_{s,t}]$ only depends on the points ${c}(s)$ and ${c}(t)$ and
not on the parametrization. Therefore, we can adjust the speed of
the parametrization so that $[{c}_{s,t}]/(t-s)$ remains in a ball
centered at $0$. This shows that $\cC (c)$ is not empty. Adjusting
the speed of the parametrization we equally get that it contains
elements that are not $0$, provided that the set $\{[c_{s,t}] ;
s<t\}$ is not bounded in $H_1(M, \ZZ )$. Certainly, if
$\{[c_{s,t}] ; s<t\}$ is bounded, all the cluster sets are reduced
to $\{ 0\}$. Observe also that if $a\in \cC (c)$ then any multiple
$\lambda a$, $\lambda>0$, belongs to $\cC (c)$, by considering the new
parametrization with velocity multiplied by $\lambda$. So $\cC
(c)$ is a cone in $H_1(M ,\RR )$.

Now we prove that $\cC (c) $ is closed. Let $a_n\in \cC (c)$ with
$a_n\to a \in H_1(M,\RR )$. For each $n$ we can choose a
parametrization of $c$, say $c^{(n)}=\tilde{c}\circ \psi_n$ (here
$\tilde{c}$ is a fixed parametrization and $\psi_n$ is an orientation
preserving homeomorphism of $\RR$), and parameters $s_n$ and $t_n$
such that $||[c_{s_n,t_n}^{(n)}]-a|| \leq 1/n$ (considering any
fixed norm in $H_1(M,\RR )$). For each $n$ we can choose $t_n$ as
large as we like, and $s_n$ negative as we like. Choose them
inductively such that $(t_n)$ and $(\psi_n(t_n))$ are both increasing
sequences converging to $+\infty$, and $(s_n)$ and $(\psi_n(s_n))$
are both decreasing sequences converging to $-\infty$. Construct a
homeomorphism $\psi$ of $\RR$ with $\psi(t_n)=\psi_n(t_n)$ and
$\psi(s_n)=\psi_n(s_n)$. It is clear that $a$ is obtained as
Schwartzman limit for the parametrization $\tilde{c}\circ \psi$ at
parameters $s_n,t_n$.

The proofs for $\cC_+(c)$ and $\cC_-(c)$ are similar.
\end{proof}

\begin{remark} \label{rem:8.16}
The image of these cluster sets in the projective space $\PP
H_1(M, \RR)$ is not necessarily connected: On the torus
$M=\TT^2=\RR^2/\ZZ^2$, choose a curve in $\RR^2$ that oscillates
between the half $y$-axis $\{ y>0\}$ and the half $x$-axis $\{
x>0\}$, remaining in a small neighborhood of these axes and being
unbounded for $t\to +\infty$,  and  being bounded when $s\to -\infty$.
%similarly for $s\to -\infty$ with the other half axes $\{y<0\}$
%and $\{ x<0\}$.
Then its Schwartzman cluster consists of two lines through $0$ in
$H_1(\TT^2, \RR)\cong \RR^2$, and its projection in the projective
space consists of two distinct points.
\end{remark}

\begin{remark}\label{rem:8.16bis}
  Let $c$ be a parametrized Schwartzman asymptotic $1$-cycle, and consider the
  unparametrized oriented curve defined by $c$, denoted by $\bar c$. Assume
  that the asymptotic Schwartzman homology class is $a=[c]\neq 0$. Then
   $$
   \cC_\pm(\bar c)=\cC(\bar c)=\RR_{\geq 0}\cdot a\,,
   $$
  as a subset of $H_1(M,\RR)$. This follows since any
  parametrization of $\bar c$ is of the form $c'=c\circ \psi$, where $\psi:\RR\to \RR$
  is a positively oriented homeomorphism of $\RR$. Then
   \begin{equation} \label{eqn:esa}
   \frac{c'_{s,t}}{t-s} =\frac{c_{\psi(s),\psi(t)}}{\psi(t)-\psi(s)} \cdot
   \frac{\psi(t)-\psi(s)}{t-s}\, .
   \end{equation}
   The first term in the right hand side tends to $a$ when $t\to +\infty$, $s\to -\infty$. If the
   left hand side is to converge, then the second term in the right hand side
   stays bounded. After extracting a subsequence, it converges to some $\lambda\geq 0$.
   Hence (\ref{eqn:esa}) converges to $\lambda\, a$.
\end{remark}

\medskip

We define now the notion of asymptotically homotopic curves.

\begin{definition}\textbf{\em (Asymptotic homotopy)} \label{def:8.18}
Let $c_0, c_1: \RR \to M$ be two parametrized curves. They are
asymptotically homotopic if there exists a continuous family
$c_u$, $u\in[0,1]$, interpolating between $c_0$ and $c_1$, such
that
 $$
 c: \RR \times [0,1] \to M \, , \ c(t,u)=c_u(t)\, ,
 $$
satisfies that $\delta_t(u)=c(t,u)$, $u\in [0,1]$  is rectifiable
with
 \begin{equation} \label{eqn:star}
 l(\delta_t)=o(|t|) \, .
  \end{equation}

Two oriented curves are asymptotically homotopic if they have
orientation preserving parametrizations that are asymptotically
homotopic.
\end{definition}

\begin{proposition}\label{prop:8.19}
If $c_0$ and $c_1$ are asymptotically homotopic parametrized curves
then their cluster sets coincide:
 \begin{align*}
 \cC_\pm (c_0)&=\cC_\pm (c_1)\, , \\
 \cC_b (c_0)&=\cC_b (c_1)\, , \\
 \cC (c_0)&=\cC (c_1)\, .
 \end{align*}

If $c_0$ and $c_1$ are asymptotically homotopic oriented curves
then their cluters sets coincide:
\begin{align*}
\cC_\pm (c_0)&=\cC_\pm (c_1)\, , \\
\cC (c_0)&=\cC (c_1)\, .
\end{align*}
\end{proposition}

% Note that we could have imposed the stronger restriction of
%uniform continuity of $c$ in the definition. Then if the curves are
%immersed and their limit set is a solenoid, by uniform continuity,
%we would have an extension to this limit set. As stated this
%extension has no reason to exist.

\begin{proof}
For parametrized curves  we have
$$
[c_{0,s,t}]=[c_{1,s,t}]+[\d_s-\g_{1,s,t}-\d_t+\g_{0,s,t}] \, .
$$
The length of the displacement by the
homotopy is bounded by (\ref{eqn:star}), so
$$
l(\d_s-\g_{1,s,t}-\d_t+\g_{0,s,t})=l(\g_{1,s,t})+l(\g_{0,s,t})+o(|t|+|s|)
\, ,
$$
thus
 $$
 \frac{[c_{0,s,t}]}{t-s}=\frac{[c_{1,s,t}]}{t-s} +o(1) \, .
 $$

For non-parametrized curves, the homotopy between two particular
parametrizations yields a one-to-one correspondence between
points in the curves
 $$
 c_0(t)\mapsto c_1(t) \, .
 $$
Using this correspondence, we have a correspondence of pairs of
points $(a,b)=(c_0(s),c_0(t))$ with pairs of points
$(a',b')=(c_1(s),c_1(t))$. Thus if the sequence of pairs of points
$(a_n,b_n)$ gives a cluster value for $c_0$, then the
corresponding sequence $(a'_n,b'_n)$ gives a proportional cluster
value, since (with obvious notation)
 $$
 [c_{0,a_n,b_n}]=[c_{1,a'_n,b'_n}] +O(1) \, .
 $$
So we can always normalize the speed of the parametrization of
$c_1$ in order to assure that the limit value is the same. This
proves that the clusters sets coincide.
\end{proof}

%%%%%%%%%%%%%%%%%%%%%%%%%%%%%%%%%%%%%%%%%%%%%%%%%%%%%%%%%%
\section{Calibrating functions}\label{sec:calibrating}
%%%%%%%%%%%%%%%%%%%%%%%%%%%%%%%%%%%%%%%%%%%%%%%%%%%%%%%%%%

Let $M$ be a $C^\infty$ smooth compact manifold. We define now the
notion of calibrating function.

Let $\pi:\tilde{M}\to M$ be the universal cover of $M$ and let
$\Gamma$ be the group of deck transformations of the cover.

Fix a point $\tilde{x}_0\in \tilde{M}$ and $x_0=\pi(\tilde{x}_0)$.
There is a faithful and transitive action of $\Gamma$ in the fiber
$\pi^{-1}(x_0)$ induced by the action of $\Gamma$ in $\tilde {M}$,
and we have a group isomorphism $\Gamma \cong \pi_1(M,x_0)$. Thus
from the group homomorphism
 $$
 \pi_1(M,x_0)\to H_1(M,{\ZZ}) \, ,
 $$
we get a group homomorphism
 $$
 \rho:\Gamma \to H_1(M,{\ZZ}) \, .
 $$

\begin{definition} \textbf{\em (Calibrating function)}
\label{def:calibrating} A map $\Phi: \tilde{M} \to H_1(M,{\RR})$
is a calibrating function if the diagram
 $$
 \begin{array}{ccc}
  \Gamma\cong\pi_1(M,x_0 )  & \hookrightarrow  & \tilde{M}  \quad \\
   \rho \downarrow  & & \downarrow \Phi\\
   H_1(M,{\ZZ}) & \to & H_1(M,{\RR})
 \end{array}
 $$
is commutative and $\Phi$ is equivariant for the action of
$\Gamma$ on $\tilde {M}$, i.e. for any $g\in \Gamma$ and $\tilde
x\in \tilde {M}$,
 $$
 \Phi (g\cdot \tilde x)=\Phi (\tilde x )+\rho (g) \, .
 $$

 If $\tilde x_0 \in \tilde M$ we say that the calibrating function
 $\Phi$ is associated to $\tilde x_0$ if $\Phi (\tilde x_0 )=0$.
\end{definition}

\begin{proposition}\label{prop:calibrating-existence}
There are smooth calibrating functions associated to any point
$\tilde x_0 \in\tilde M$.
\end{proposition}

\begin{proof}
Fix a smooth  non-negative function $\varphi:\tilde{M}\to {\RR}$
with compact support $K={\overline U}$ with $U=\{\varphi >0\}$
such that $\pi(U)=M$. Moreover, we can request that $U\cap
\pi^{-1}(x_0)=\{\tilde{x}_0\}$.

For any $g_0\in\Gamma$, define $\varphi_{g_0}(\tilde{x})=
\varphi(g_0^{-1} \cdot \tilde{x})$. The support of $\varphi_{g_0}$
is $g_0\, K$, and $(g_0\, K)_{g_0\in \Gamma}$  is a locally finite
covering of $\tilde{M}$, as follows from the compactness of $K$.
Set
 $$
 \psi_{g_0}(\tilde{x}) := \frac{\varphi_{g_0}(\tilde{x})}{\sum_{g\in
 \Gamma} \varphi_{g}(\tilde{x})} \ \, .
 $$
Then $\psi_{g_0}(\tilde{x})=\psi_e(g_0^{-1}\cdot\tilde{x})$ and
 $$
 \sum_{g\in \Gamma} \psi_g \equiv 1 \, .
 $$
Also $\psi_{g_0}$ has compact support $g_0\, K$, and it is a smooth
function since the denominator is strictly positive (because $\pi
(U)=M$) and it is at each point a finite sum of smooth
functions.

We define the map
 $$
 \Phi: \tilde{M} \to H_1(M,{\RR}) \ \, ,
 $$
by
 $$
 \Phi(\tilde{x})= \sum_{g\in \Gamma} \psi_g(\tilde{x})\, \rho(g) \, .
 $$
We check that $\Phi$ is a calibrating function:
 $$
 \begin{aligned}
 \Phi(g\cdot \tilde{x}) &= \sum_{h\in \Gamma}\psi_h(g\cdot\tilde{x})\,\rho(h)\\
    &= \sum_{h\in \Gamma}\psi_{g^{-1}h}(\tilde{x})\,(\rho(g)+\rho(g^{-1}h)) \\
    &= \sum_{h'\in \Gamma}\psi_{h'}(\tilde{x})\,\rho(g) \  + \
       \sum_{h'\in \Gamma}\psi_{h'}(\tilde{x})\, \rho(h') \\
    &= \rho(g)+ \Phi(\tilde{x})\, .
 \end{aligned}
 $$
Notice that by construction $\Phi (\tilde x_0)=0$.
\end{proof}

We note also that choosing a function $\phi$ of rapid decay, we
may do a similar construction, as long as $\sum_{g\in\Gamma}
\phi_g$ is summable (we may need to add a translation to $\Phi$ in
order to ensure $\Phi (\tilde x_0)=0$).

Observe that the calibrating property implies that for a curve
$\gamma:[a,b]\to M$, the quantity
$\Phi(\tilde{\gamma}(b))-\Phi(\tilde{\gamma}(a))$ does not depend
on the lift $\tilde \gamma$ of $\gamma$, because for another
choice $\tilde \gamma '$, we would have for some $g\in \Gamma$,
 $$
 \tilde \gamma'(a) =g\cdot\tilde \gamma (a) \, ,
 $$
and
 $$
 \tilde \gamma'(b) =g\cdot\tilde \gamma (b) \, .
 $$
Therefore
 $$
 \Phi(\tilde{\gamma}'(b))-\Phi(\tilde{\gamma}'(a))
 =\Phi(g\cdot\tilde{\gamma}(b))-\Phi(g\cdot\tilde{\gamma}(a))
 =\Phi(\tilde{\gamma}(b))-\Phi(\tilde{\gamma}(a)) \, .
 $$

This justifies the next definition.

\begin{definition}
Given a calibrating function $\Phi$, for any curve $\gamma:[a,b]\to
M$, we define
$\Phi(\gamma):=\Phi(\tilde{\gamma}(b))-\Phi(\tilde{\gamma}(a))$ for
any lift $\tilde \gamma$ of $\gamma$.
\end{definition}

\begin{proposition} \label{def:8.14}
For any loop $\gamma \subset M$ we have
 $$
 \Phi(\gamma)=[\gamma]\in H_1(M,{\ZZ})\, .
 $$
\end{proposition}

\begin{proof}
Modifying $\gamma$, but without changing its endpoints nor $\Phi
(\gamma)$ nor $[\gamma ]$, we can assume that $x_0\in \gamma$.
Since $\Gamma \cong \pi_1(M,x_0)$, let $h_0\in\Gamma$ be
the element corresponding to $\gamma$. Then $\gamma$ lifts to a curve
joining $\tilde x_0$ to $h_0\cdot \tilde x_0$, and
 $$
 \Phi(\gamma)=\Phi(h_0 \cdot \tilde{x}_0)-
 \Phi (\tilde x_0)
 =\rho(h_0)=[\gamma]\in H_1(M,{\ZZ})\, .
 $$
\end{proof}

\begin{proposition} \label{prop:8.15}
We assume that $M$ is endowed with a Riemannian metric and that the
calibrating function $\Phi$ is smooth. Then for any rectifiable
curve $\gamma$ we have
$$
|\Phi(\gamma)|\leq C \cdot l(\gamma) \, ,
$$
where $l(\gamma )$ is the length of $\gamma$, and $C>0$ is a
positive constant depending only on the metric.
\end{proposition}

\begin{proof}
The calibrating function $\Phi$ is a smooth function on $\tilde{M}$
and $\Gamma$-equivariant, hence it is bounded as well as its
derivatives. The result follows.
\end{proof}

\medskip

\begin{example}
For $M=\TT$, $\tilde{M}={\RR}$, $H_1(M,\ZZ)=\ZZ\subset \RR
=H_1(M,\RR)$, $\Gamma =\ZZ$ and $\rho : \Gamma \to H_1(M,\ZZ)$ is
given (with these identifications) by $\rho (n)=n$. We can take
$\varphi(x)= |1-x|$, for $x\in [-1,1]$, and $\varphi (x)=0$
elsewhere. Then
 $$
 \sum_{n=-\infty}^\infty \varphi(x-n)=1 \, ,
 $$
and
 $$
 \psi_n(x)=\varphi_n(x)=\varphi (x-n) \, .
 $$
Therefore we get the calibrating function
 $$
 \Phi (x)= \sum_{n=-\infty}^\infty \varphi(x-n) \, n= x \,  .
 $$
It is a smooth calibrating function (despite that $\varphi$ is not).

A similar construction works for higher dimensional tori.
\end{example}

\begin{proposition}\label{prop:asymptotic}
Let $c:{\RR}\to M$ be a $C^1$ curve. Consider two sequences $(s_n)$
and $(t_n)$ such that $s_n<t_n$, $s_n \to -\infty$, and $t_n\to
+\infty$.

Then the following conditions are equivalent:
\begin{enumerate}
 \item The limit
 $$
[c]=\lim_{n\to +\infty} \frac{[c_{s_n , t_n}]}{t_n-s_n} \in
H_1(M,{\RR})
 $$
 exists.

\medskip

 \item The limit
  $$
  [c]_{\Phi}=\lim_{n\to \infty} \frac{\Phi(c_{|[s_n,t_n]}) }{t_n-s_n}   \in H_1(M,{\RR})
  $$
  exists.

\medskip

 \item For any closed $1$-form $\alpha \in \Omega^1(M)$, the
  limit
  $$
  [c](\alpha )=\lim_{n\to \infty} \frac{1}{t_n-s_n}\int_{c([s_n , t_n])} \alpha
   $$
  exists.

\medskip

\item For any cohomology class $[\alpha ]\in H^1(M,\RR )$, the limit
  $$
  [c][\alpha ]=\lim_{n\to \infty} \frac{1}{t_n-s_n}\int_{c([s_n , t_n])} \alpha
  $$
exists, and does not depend on the closed  $1$-form $\alpha \in
\Omega^1(M)$ representing the cohomology class.

\medskip

\item For any continuous map $f:M\to \TT$, let $\widetilde{f\circ
c}:{\RR} \to {\RR}$ be a lift of $f\circ c$,  the limit
 $$
 \rho (f)=\lim_{n\to +\infty} \frac{\widetilde{f\circ c}(t_n) -
 \widetilde{f\circ c}(s_n)}{t_n-s_n}
 $$
exists.

\medskip

\item For any (two-sided, embedded, transversally oriented)
hypersurface $H\subset M$ such that all intersections
$c({\RR})\cap H$ are transverse, the limit
  $$
  [c]\cdot [H]=\lim_{n\to \infty}  \frac{\# \{u \in [s_n,t_n]\ ; \
  c(u)\in H\}}{t_n-s_n}
  $$
exists. The notation $\#$ means a signed count of intersection
points.

\end{enumerate}

When these conditions hold, we have $[c]=[c]_{\Phi}$ for any
calibrating function $\Phi$. If $\alpha \in \Omega^1 (M)$ is a
closed form, then $[c](\alpha)=[c][\alpha]=\la [c],[\alpha] \ra$.
If $f:M\to \TT$ is a continuous map and $a=f^*[dx]\in H^1(M,\ZZ)$
is the pull-back of the generator $[dx]\in H^1(\TT,\ZZ)$, and $H$
is a hypersurface such that $[H]$ is the Poincar\'{e} dual of $a$,
then $\la [c],a\ra= \rho (f)=[c]\cdot [H]$.
\end{proposition}

\begin{proof}
The equivalence of (1) and (2) follows from the properties of
$\Phi$. Let $c:{\RR}\to M$ be a curve. Then
 $$
 \Phi(c_{|[s_n,t_n]}) = \Phi([c_{s_n,t_n}])- \Phi(\gamma_{s_n,t_n}) =[c_{s_n,t_n}]+O(l(\gamma_{s_n,t_n})) \, .
 $$
Dividing by $t_n-s_n$ and passing to the limit the equivalence of
(1) and (2) follows.

\medskip

We prove that (1) is equivalent to (3). First note that
 $$
  \left| \int_{\gamma_{s_n,t_n}} \alpha  \right| \leq C
   \, l(\g_{s_n,t_n})\,||\alpha||_{C^0}\, .
 $$
We have when $t_n-s_n\to +\infty$,
$$
  \frac{1}{t_n-s_n} \int_{c([s_n,t_n])} \alpha  = \frac{1}{t_n-s_n}\int_{c_{s_n,t_n}} \alpha
  + O\left(\frac{l(\g_{s_n,t_n})}{t_n-s_n}\right)=\frac{[c_{s_n,t_n}](\alpha )}{t_n-s_n}
  +o(1) \, .
$$
and the equivalence of (1) and (3) results.

\medskip

The equivalence of (3) and (4) results from the fact that the limit
  $$
  [c](\alpha)=\lim_{n\to \infty} \frac{1}{t_n-s_n}\int_{c([s_n,t_n])}
  \alpha
  $$
does not depend on the representative of the cohomology class
$a=[\alpha]$. If $\beta=\alpha+d\phi$, with $\phi:M\to {\RR}$
smooth, then $[c](\alpha)=[c](\beta)$ since
  $$
  [c](d\phi)=\lim_{n\to \infty} \frac{1}{t_n-s_n} \int_{c([s_n,t_n])} d\phi
  =\lim_{n\to \infty}  \frac{\phi(c(t_n))-\phi(c(s_n))}{t_n-s_n} \to 0,
  $$
since $\phi$ is bounded. Also $[c][\alpha]=[c](\alpha)$.

\medskip

We turn now to (4) implies (5). First note that there is an
identification $H^1(M,{\ZZ}) \cong [M,K({\ZZ},1)]=[M, \TT]$, where any
cohomology class $[\alpha]\in H^1(M, {\ZZ})$ is associated to a
(homotopy class of a) map $f:M\to \TT$ such that
$[\alpha]=f^*[\TT]$, where $[\TT]\in H^1(\TT,\ZZ)$ is the
fundamental class. To prove (5), assume first that $f$ is smooth.
With the identification $\TT={\RR}/{\ZZ}$, the class $f^*(d
x)=df\in \Omega^1(M)$ represents $[\alpha]$. Therefore
 \begin{equation} \label{eqn:masstar}
 \begin{aligned}
 &\frac{\widetilde{f\circ c}(t_n) -
 \widetilde{f\circ c}(s_n)}{t_n-s_n}= \frac{1}{t_n-s_n} \int_{[s_n,t_n]} d(f\circ
 c)=\\
 & = \frac{1}{t_n-s_n} \int_{[s_n,t_n]} (df)(c') =
 \frac{1}{t_n-s_n} \int_{c([s_n,t_n])} df \, ,
 \end{aligned}
  \end{equation}
and from the existence of the limit in (4) we get the limit in (5)
that we identify as
 $$
 \rho (f)= [c][df] \, .
 $$
If $f$ is only continuous, we approximate it by a smooth function,
which does not change the limit in (5).

Conversely, if (5) holds, then any integer cohomology class admits
a representative of the form $\alpha =df$, where $f: M\to \TT$ is
a smooth map. Then using (\ref{eqn:masstar}) we have
 $$
 \frac{1}{t_n-s_n} \int_{c([s_n,t_n])} \alpha \to \rho(f)\, .
 $$
So the limit in (4) exists for $\alpha=df$.
This implies that the limit in (4) exists for any closed $\alpha \in
\Omega^1 (M)$, since $H^1(M,\ZZ )$ spans $H^1(M,\RR )$.

\medskip

We check the equivalence of (5) and (6). First, let us see that
(6) implies (5). As before, it is enough to prove (5) for a smooth
map $f:M\to \TT$. Let $x_0\in \TT$ be a regular value of $f$, so
that $H=f^{-1}(x_0)\subset M$ is a smooth (two-sided)
hypersurface. Then $[H]$ %\in H_{n-1}(M,{\ZZ})$
represents the
Poincar\'{e} dual of $[df]\in H^1(M,{\ZZ})$. Choose $x_0$ such that it
is also a regular value of $f\circ c$, so all the
intersections of $c(\RR)$ with $H$ are transverse. Now for any $s<t$, %such that
%$f(c(s))\neq x_0$ and $f(c(t))\neq x_0$,
 $$
 [c_{s,t}]\cdot [H]= \# c([s,t]) \cap H + \# \gamma_{s,t} \cap
 H\, ,
 $$
where $\#$ denotes signed count of intersection points (we may
assume that all intersections of $\gamma_{s,t}$ and $H$ are
transverse, by a small perturbation of $\g_{s,t}$; also we do not
count the extremes of $\gamma_{s,t}$ in $\# \gamma_{s,t} \cap H$
in case that either $c(s)\in H$ or $c(t)\in H$).

Now
 $$
 \# c ([s,t])\cap H =[\widetilde{f\circ c}(t)] + [ -
 \widetilde{f\circ c}(s)]=\widetilde{f\circ c}(t) -
 \widetilde{f\circ c}(s) + O(1),
 $$
where $[\cdot]$ denotes the integer part, and $| \# \gamma_{s,t}
\cap H|$ is bounded by the total variation of $\widetilde{f\circ
\gamma_{s,t}}$, which is bounded by the maximum of $df$ times the
total length of $\gamma_{s,t}$, which is $o (t-s)$ by assumption.
Hence
 $$
 \lim_{n\to +\infty} \frac{\widetilde{f\circ c}(t_n) -
 \widetilde{f\circ c}(s_n)}{t_n-s_n} =
 \lim_{n\to +\infty}  \frac{\# c([s_n,t_n])\cap H}{t_n-s_n}
 $$
exists.

Conversely, if (5) holds, consider a two-sided embedded
topological hypersurface $H\subset M$. Then there is a collar
$[0,1]\times H$ embedded in $M$ such that $H$ is identified with
$\{\frac12 \}\x H$. There exists a continuous map $f:M\to \TT$
such that $H=f^{-1}(x_0)$ for $x_0=\frac12\in \TT$, constructed by
sending $[0,1]\times H \to [0,1]\to \TT$  and collapsing the
complement of $[0,1]\times H$ to $0$.

Now if all intersections of $c(\RR)$ and $H$ are transverse, that
means that for any $t\in\RR$ such that $c(t)\in H$, we have that
$c(t-\epsilon)$ and $c(t+\epsilon)$ are at opposite sides of the
collar, for $\epsilon>0$ small (the sign of the intersection point
is given by the direction of the crossing). So $f(c(s))$ crosses
$x_0$ increasingly or decreasingly (according to the sign of the
intersection). Hence
 $$
 \frac{\# \{u \in [s_n,t_n] \ ; \ c(u)\in H\}}{t_n-s_n} =
 \frac{\widetilde{f\circ c}(t_n) - \widetilde{f\circ
 c}(s_n)}{t_n-s_n}+o(1).
 $$
The required limit exists.
\end{proof}

\medskip

\begin{remark}
Proposition \ref{prop:asymptotic} holds if we only assume the
curve $c$ to be rectifiable.
\end{remark}

\medskip

\begin{corollary}\label{prop:asymptotic2}
Let $c:{\RR}\to M$ be a $C^1$  curve. The following conditions are
equivalent:

\begin{enumerate}
 \item The curve $c$ is a Schwartzman asymptotic cycle.

\medskip

 \item The limit
  $$
  \lim_{t\to +\infty \atop s\to -\infty } \frac{\Phi(c_{|[s,t]})}{t-s}   \in H_1(M,{\RR})
  $$
  exists.

\medskip

 \item For any closed $1$-form $\alpha \in \Omega^1(M)$, the
  limit
  $$
  \lim_{t\to +\infty \atop s\to -\infty } \frac{1}{t-s}\int_{c([s,t])} \alpha
   $$
  exists.

\item For any cohomology class $[\alpha ]\in H^1(M,\RR )$, the limit
  $$
  [c][\alpha ]=\lim_{t\to +\infty \atop s\to -\infty} \frac{1}{t-s}\int_{c([s , t])} \alpha
   $$
  exists, and does not depend on the closed  $1$-form $\alpha \in
  \Omega^1(M)$ representing the cohomology class.

\medskip

 \item For any continuous map $f:M\to \TT$, let $\widetilde{f\circ c}:{\RR}
 \to {\RR}$ be a lift of $f\circ c$, we have that the limit
 $$
 \lim_{t\to +\infty \atop s\to -\infty } \frac{\widetilde{f\circ c}(t) -
 \widetilde{f\circ c}(s)}{t-s}
 $$
 exists.

\medskip

 \item For a (two-sided, embedded, transversally oriented) hypersurface $H\subset M$ such that all
intersections $c({\RR})\cap H$ are transverse, the limit
  $$
  \lim_{t\to +\infty \atop s\to -\infty }  \frac{\# \{u \in [s,t] \, ; \, c(u)\in H\}}{t-s}
  $$
  exists.
\end{enumerate}

When $c$ is a Schwartzman asymptotic cycle, we have
$[c]=[c]_{\Phi}$ for any calibrating function $\Phi$. If $\alpha
\in \Omega^1 (M)$ is a closed form then
 $$
 [c](\alpha)=[c][\alpha]=\la [\alpha],[c]\ra\,.
 $$
If $f:M\to \TT$ and
$a=f^*[dx]\in H^1(M,\ZZ)$, where $[dx]\in H^1(\TT,\ZZ)$ is the
generator, and $H \subset M$ is a hypersurface such that $[H]$ is
the Poincar\'e dual of $a$, then we have
 $$
 \la [c],[\alpha] \ra =\rho (f)=[c]\cdot [H] \, .
 $$
\end{corollary}

%%%%%%%%%%%%%%%%%%%%%%%%%%%%%%%%%%%%%%%%%%%%%%%%%%%%%%%%%%%%%%%%%%
\section{Schwartzman $1$-dimensional cycles}\label{sec:Schwartzman-cycles}
%%%%%%%%%%%%%%%%%%%%%%%%%%%%%%%%%%%%%%%%%%%%%%%%%%%%%%%%%%%%%%%%%%

We assume that $M$ is a compact %oriented
$C^\infty$ Riemannian manifold, with Riemannian metric $g$.

\begin{definition} \textbf{\em (Schwartzman representation of homology classes)}
\label{def:representation-Schwartzman-1-solenoid} Let $f:S\to M$ be
an immersion in $M$ of an oriented $1$-solenoid $S$. Then $S$ is a
Riemannian solenoid with the pull-back metric $f^* g$.

\begin{enumerate}
\item If $S$ is endowed with a transversal measure
 $\mu=(\mu_T)\in \cM_\cT (S)$, the immersed measured solenoid $f:S_\mu\to M$
 represents a homology class $a\in H_1(M, \RR)$ if for
 $(\mu_T)$-almost all leaves $c:{\RR} \to S$, parametrized
 positively and by arc-length, we have that $f\circ c$ is a
 Schwartzman asymptotic $1$-cycle with $[f\circ c]=a$.

\item The immersed solenoid $f:S\to M$ fully represents a homology
class $a\in
 H_1(M, \RR)$ if for all leaves $c:{\RR} \to S$,
 parametrized positively and by arc-length, we have that $f\circ c$ is a
 Schwartzman asymptotic $1$-cycle with $[f\circ c]=a$.
\end{enumerate}

\end{definition}

Note that if $f:S\to M$ fully represents an homology class $a\in
H_1(M, \RR)$, then for all oriented leaves $c\subset  S$, we have
that $f \circ c$ is a Schwartzman asymptotic cycle and
 $$
 \cC_+(f \circ c)=\cC_-(f \circ c)%=\cC_b(c)
 =\cC(f \circ c)=\RR_{\geq 0}\cdot a \subset H_1(M,\RR) \, ,
 $$
by remark \ref{rem:8.16bis}.

Observe that contrary to what happens with Ruelle-Sullivan cycles,
we can have an immersed solenoid fully representing an homology
class without the need of a transversal measure on $S$.

\begin{definition} \textbf{\em (Cluster of an immersed solenoid)}
\label{def:clusters-Schwartzman-1-solenoid}
Let $f:S\to M$ be an immersion in $M$ of an oriented $1$-solenoid
$S$. The homology cluster of $(f,S)$, denoted by $\cC
({{f,S}})\subset H_1(M,\RR)$, is defined as the derived set of %cluster set of
$([(f\circ c)_{s,t}]/(t-s))_{c,t\to \infty, s\to -\infty}$, taken over all
images of orientation preserving parametrizations $c$ of all leaves of
$S$, and $t\to +\infty$ and $s\to -\infty$. Analogously, we define
the corresponding positive and negative clusters.

The Riemannian cluster of $(f,S)$, denoted by  $\cC^g ({{f,S}})$,
is defined in a similar way, using arc-length orientation
preserving parametrizations. Analogously, we define the positive,
negative and balanced Riemannian clusters.
%$$
%\cC^{f^*g} (S)=\bigcup_{c\subset S} \cC^{f^*g} (c) \, ,
%$$
%where $\cC^{g'}(c)$ denotes the cluster of $c$ for the positive
%arc-length parametrization.
\end{definition}

\medskip

As in section \ref{sec:1-schwartzman}, we can prove with arguments
analogous to those of propositions \ref{prop:Riemannian-clusters}
and \ref{prop:8.15.si}\ :

\begin{proposition} \label{prop:10.3}
The homology clusters $\cC({{f,S}})$, $\cC_\pm ({{f,S}})$ are
non-empty, closed cones of $H_1(M,\RR)$. If these cones are
non-degenerate, their images in $\PP H_1(M,\RR)$ are non-empty and
compact sets.

The Riemannian homology clusters $\cC^g ({{f,S}})$, $\cC_\pm^g
({{f,S}})$ are non-empty, compact and connected subsets of
$H_1(M,\RR)$.
\end{proposition}

The following proposition is clear, and gives the relationship with
the clusters of the images by $f$ of the leaves of $S$.

\begin{proposition}
Let $f:S\to M$ be an immersion in $M$ of an oriented $1$-solenoid
$S$. We have
 $$
 \bigcup_{c\subset S} \cC (f\circ c) \subset \cC ({{f,S}}) \, ,
 $$
where the union runs over all parametrizations of leaves of $S$.
We also have
 $$
 \bigcup_{c\subset S} \cC_\pm (f\circ c) \subset \cC_\pm ({{f,S}}) \, ,
 $$
and
 $$
 \bigcup_{c\subset S} \cC_b (f\circ c) \subset \cC_b ({{f,S}}) \, .
 $$
And similarly for all Riemanniann clusters with $\cC_* (f\circ
c)$ denoting the Schwartzman clusters for the arc-length
parametrization.
\end{proposition}

We recall that given an immersion $f:S\to M$ of an oriented
$1$-solenoid, $S$ becomes a
Riemannian solenoid and theorem \ref{thm:transverse-riemannian}
gives a one-to-one correspondence
between the space of transversal measures (up to scalar
normalization) %and $\cM_\cL (S)$,
and the space of daval measures,
 $$
 \overline\cM_\cT (S) \cong \cM_\cL (S) \, .
 $$

Moreover, in the case of $1$-solenoids that we consider here, %if
%they contain no loops,
they do satisfy the controlled growth condition of definition
\ref{def:controlled-growth}. Therefore all Schwartzman measures
desintegrate as length on leaves by theorem
\ref{thm:desintegration-Schwartzman-measures}.

Giving any transversal measure $\mu$ we can consider the associated
generalized current $(f,S_\mu)\in \cC_k(M)$.

\begin{definition} We define the Ruelle-Sullivan map
 $$
 \Psi : \cM_\cT (S) \to H_1(M,\RR)
 $$
by
 $$
 \mu \mapsto \Psi (\mu )=[f,S_\mu] \, .
 $$

The Ruelle-Sullivan cluster cone of $(f,S)$ is the image of $\Psi$
 $$
 \cC_{RS} ({{f,S}}) =\Psi (\cM_\cT (S))=\left \{ [f,S_\mu] ; \mu
 \in \cM_\cT (S)\right \} \subset H_1(M,\RR ) \, .
 $$

The Ruelle-Sullivan cluster set is
 $$
 \PP\cC_{RS} ({{f,S}}) \cong \left \{[f,S_\mu] ; \mu \in \cM_\cL (S)\right \} \subset
 H_1(M,\RR ) \, ,
 $$
i.e. using transversal measures which are normalized (using the Riemannian
metric of $M$).
\end{definition}

\medskip

\begin{proposition} \label{prop:11.6}
Let $\cV_\cT(S)$ be the set of all signed measures, with finite
absolute measure and invariant by holonomy, on the solenoid $S$.
The Ruelle-Sullivan map $\Psi$ extended by linearity to
$\cV_\cT(S)$ is a linear continuous operator,
 $$
 \Psi : \cV_\cT (S) \to H_1(M,\RR) \, .
 $$
\end{proposition}

\begin{proof}
Coming back to the definition of generalized current, it is
clear that $\mu\mapsto [f,S_\mu]$ is linear in flow-boxes,
therefore globally. It is also continuous because if $\mu_n\to
\mu$, then $[f,S_{\mu_n}]\to [f,S_\mu ]$ as can be seen in a fixed
flow-box covering of $S$.
\end{proof}

\begin{corollary}
The Ruelle-Sullivan cluster $\cC_{RS} ({{f,S}})$ is a non-empty,
convex, compact cone of $H_1(M,\RR)$. Extremal points of the
convex set $\cC_{RS} ({{f,S}})$ come from the generalized currents
of ergodic measures in $\cM_\cL (S)$.
\end{corollary}

\begin{proof}
Since $\cM_\cL (S)$ is non-empty, convex and compact set, its
image by the continuous linear map $\Psi$ is also a non-empty,
convex and compact set. Any extremal point of $\cC_{RS} ({{f,S}})$
must have an extremal point of $\cM_\cL (S)$ in its pre-image, and
these are the ergodic measures in $\cM_\cL (S)$ (according to the
identification of $\cM_\cL (S)$ to $\overline\cM_\cT (S)$ and by
proposition 5.11 in \cite{MPM1}).
\end{proof}

It is natural to investigate the relation between the Schwartzman
cluster and the Ruelle-Sullivan cluster.

\begin{theorem} \label{thm:11.8}
Let $S$ be a $1$-solenoid. For any immersion $f:S\to M$ we have
 $$
 \bigcup_{c\subset S} \cC (f\circ c) \subset \cC_{RS} ({{f,S}}) \, .
 $$
\end{theorem}

\begin{proof}
It is enough to prove the theorem for minimal solenoids, since
each leaf $c\subset S$ is contained in a minimal solenoid
$S_0\subset S$, and
 $$
 \cC (f\circ c) \subset \cC_{RS} (f,S_0)\subset
 \cC_{RS} ({{f,S}}) \, .
 $$
The last inclusion holds because if $\mu$ is a transversal measure
for $S_0$, then it defines a transversal measure $\mu'$ for $S$,
which is clearly invariant by holonomy. Now the generalized
currents coincide, $(f,S_{\mu'})=(f,S_{0,\mu})$, as can be seen by
in a fixed flow-box covering of $S$. Therefore, the Ruelle-Sullivan
homology classes are the same, $[f,S_{\mu'}]=[f,S_{0,\mu}]$.

The statement for minimal solenoids follows from theorem
\ref{thm:11.9} below.
\end{proof}

\begin{theorem} \label{thm:11.9}
Let $S$ be a minimal $1$-solenoid. For any immersion $f:S\to M$ we have
 $$
 \cC({{f,S}})\subset \cC_{RS} ({{f,S}}) \, .
 $$
\end{theorem}

%{\bf Observation.}%
%The proof only uses that $S$ has at each point (or in a dense set
%of points) arbitrarily small global transversals. But this is
%indeed equivalent to minimality.

\begin{proof}
Consider an element $a\in \cC({{f,S}})$ obtained as limit of a
sequence $([(f\circ c_{n})_{s_n,t_n}])$, where $c_n$ is an positively
oriented parametrized leaf of $S$ and $s_n<t_n$, $s_n\to -\infty$,
$t_n\to \infty$. The points $(c_n(t_n))$ must accumulate a
point $x\in S$, and taking a subsequence, we can
assume they converge to it. Choose a
small local transversal $T$ of $S$ at this point, such that
$f(T)\subset B$ where $B\subset M$ is a contractible ball in $M$.
By minimality,
the return map $R_T: T\to T$ is well defined.

Note that we may assume that $\bar T\subset T'$, where $T'$ is also a local
transversal. By compactness of $\bar T$, the return time for $R_{T'}:T'\to T'$
of any leaf, measured with the arc-length parametrization, for any
$x\in \bar T$, is universally bounded. Therefore we can adjust
the sequences $(s_n)$ and $(t_n)$ such that $c_n(s_n)\in \bar T$ and
$c_n(t_n) \in \bar T$, by changing each term by an amount $O(1)$.
Now, after further taking a subsequence, we can arrange that
$c_n(s_n),c_n(t_n) \in T$.

Taking again a subsequence if necessary we can assume that we have a
Schwartzman limit of the measures $\mu_n$ which correspond to the
arc-length on $c_n([s_n,t_n])$ normalized with total mass $1$. The
limit measure $\mu$ desintegrates on leaves because of theorem
\ref{thm:desintegration-Schwartzman-measures}, so it defines a
trasnversal measure $\mu$.

The transversal measures corresponding to $\mu_n$ are atomic,
supported on $T\cap c_n([s_n, t_n))$, assigning the weight
$l([x,R_T(x)])$ to each point in $T\cap c_n([s_n, t_n))$. The
transversal measure corresponding to $\mu$ is its normalized
limit. For each $1$-cohomology class, we may choose a closed
$1$-form $\omega$ representing it and vanishing on $B$ (this is
so because $H^1(M,B)=H^1(M)$, since $B$ is contractible). Assume
that we have constructed $[(f\circ c_{n})_{s_n,t_n}]$ by using
$\gamma_{n,s_n,t_n}$ inside $B$. So
% $$
% \la [f,S_{\mu_n}],\omega' \ra = \la \frac{[c_{n,s_n,t_n}]}{t_n-s_n},
% [\omega']\ra = {1\over t_n-s_n} \int_{c_n([s_n,t_n])} \omega'
% \, .
% $$
 $$
 \la [f,S_{\mu_n}],\omega \ra = \int_S f^*\omega \, d\mu_n
  =\int_{f\circ c_n([s_n,t_n])} \omega = \la [(f\circ c_{n})_{s_n,t_n}], [\omega]\ra
 \, ,
 $$
thus
 \begin{equation*}%\label{eqn:2}
 \la [f,S_\mu],[\omega] \ra = \lim_{n\to \infty}
 \frac{1}{t_n-s_n} \la [f,S_{\mu_n}],\omega \ra =
 \lim_{n\to \infty} \la \frac{[(f\circ c_{n})_{s_n,t_n}]}{t_n-s_n}, [\omega]\ra =\la
 a,[\omega]\ra\, .
 \end{equation*}
Thus the generalized current of the limit measure coincides with
the Schwartzman limit.
\end{proof}

We use the notation $\partial^* C$ for the extremal points of a
compact convex set $C$. For the converse result, we have:

\begin{theorem} Let $S$ be a minimal solenoid and an immersion $f:S\to M$. We have
 $$
 \partial^* \cC_{RS} ({{f,S}}) \subset \bigcup_{c\subset S} \cC (f\circ
 c)\subset \cC({{f,S}}) \, .
 $$
\end{theorem}

\begin{proof}
We have seen that the points in $\partial^* \cC_{RS} ({{f,S}})$
come from ergodic measures in $\cM_\cL (S)$ by the Ruelle-Sullivan
map. Therefore it is enough to prove the following theorem that
shows that the Schwartzman cluster of almost all leaves is reduced
to the generalized current for an ergodic $1$-solenoid.
\end{proof}

\begin{theorem}\label{thm:Ruelle-1-solenoid}
Let $S$ be a minimal $1$-solenoid endowed with an ergodic measure
$\mu \in \cM_\cL (S)$. Consider an immersion
$f:S\to M$. Then for $\mu$-almost all leaves $c\subset S$ we have that
$f\circ c$ is a Schwartzman asymptotic $1$-cycle and
 $$
 [f\circ c]=[f,S_\mu] \in H_1(M, \RR)\, .
 $$
Therefore the immersion  $f:S_\mu\to M$ represents its Ruelle-Sullivan
homology class.

In particular, this homology class is independent of the metric $g$
on $M$ up to a scalar factor.
\end{theorem}

\begin{proof}
The proof is an application of Birkhoff's ergodic theorem. Choose a
small local transversal $T$ such that $f(T)\subset B$, where
$B\subset M$ is a small contractible ball. Consider the associated
Poincar\'e first return map $R_T: T\to T$. Denote by $\mu_T$ the
transversal measure supported on $T$.

For each $x\in T$ we consider $\varphi_T (x)$ to be the homology
class in $M$ of the loop image by $f$  of the leaf $[x,R_T(x)]$
closed by a segment in $B$ joining $x$ with $R_T(x)$. In this way
we have defined a measurable map %continuous map
 $$
 \varphi_T : T \to H_1(M, \ZZ) \, .
 $$
Also for $x\in S$, we denote by $l_T(x)$ the length of the leaf
joining $x$ with its first impact on $T$ (which is  $R_T(x)$ for
$x\in T$).
We have then an upper semi-continuous map %continuous map
 $$
 l_T: S \to \RR_+ \, .
 $$
Therefore $l_T$ is bounded by compactness of $S$. In particular,
$l_T$ is bounded on $T$ and thefore in $L^1(T,\mu_T)$. The
boundedness of $l_T$ implies also the boundedness of $\varphi_T$ by
lemma \ref{lem:9.1}.

Consider $x_0\in T$ and its return points $x_i=R_T^i(x_0)$. Let
$0<t_1 < t_2 < t_3 < \ldots $ be the times of return for the
positive arc-length parametrization. We have
 $$
 t_{i+1}-t_i =l_T(x_{i}) \, .
 $$
Therefore
 $$
 t_n=\sum_{i=0}^{n-1} (t_{i+1}-t_i) =\sum_{i=0}^{n-1} l_T\circ R_T^i
 (x_0) \, ,
 $$
and by Birkhoff's ergodic theorem
 $$
 \lim_{n\to +\infty } \frac{1}{n} t_n =\int_T l_T(x) \ d\mu_T (x)
 =\mu (S) =1\, .
 $$

Now observe that, by contracting $B$, we have
 \begin{align*}
 [f\circ c_{0,t_n}] &= [f\circ c_{0,t_1}]+ [f\circ c_{t_1,t_2}]+\ldots
  +[f\circ c_{t_{n-1},t_n}] \\
  &=\varphi_T(x_0) +\varphi_T\circ R_T
 (x_0)+\ldots +\varphi_T \circ R_T^{n-1}(x_0) \, .
 \end{align*}
We recognize a Birkhoff's sum and  by Birkhoff`s ergodic theorem we
get the limit
 $$
 \lim_{n\to +\infty } \frac{1}{n} [f\circ c_{0,t_n}] =\int_T \varphi_T (x) \
 d\mu_T (x) \in H_1(M,\RR)\, .
 $$
Finally, putting these results together,
 $$
 \lim_{n\to +\infty } \frac{1}{t_n} [f\circ c_{0,t_n}]=\lim_{n\to +\infty }
 \frac{[f\circ c_{0,t_n}]/n}{t_n/n} = \frac{ \int_T \varphi_T (x) \ d\mu_T
 (x)}{\int_T l_T(x) \ d\mu_T (x) } = %{1\over \mu (S)}
 \int_T \varphi_T (x) \ d\mu_T
 (x)   \, .
 $$

Let us see that this equals the generalized current. Take a
closed $1$-form $\omega\in \Omega^1(M)$, which we can assume to
vanish on $B$. Then
 $$
 \la [f,S_\mu],\omega\ra = \int_T \left(\int_{[x,R_T(x)]} f^*\omega
 \right) d\mu_T(x) = \int_T \la \varphi_T(x),\omega\ra d\mu_T(x)\,
 ,
 $$
and so
 $$
 [f,S_\mu] = \int_T \varphi_T(x)\, d\mu_T(x)\, .
 $$

Observe that so far we have only proved that $\cC_+^g(f\circ
c)=\{[f,S_\mu]\}$ for almost all leaves $c\subset S$. Considering
the reverse orientation, the result follows for the negative
clusters, and finally for the whole cluster of almost all leaves.

The last statement follows since $[f,S_\mu]$ only depends on
$\mu\in\cM_\cT(S)$, which is independent of the metric up to scalar
factor, thanks to the isomorphism of theorem \ref{thm:transverse-riemannian}.
\end{proof}

Therefore for a minimal oriented ergodic $1$-solenoid, the generalized
current coincides with the Schwartzman asymptotic homology
class of almost all leaves. It
is natural to ask when this holds for all leaves, i.e. when the
solenoid fully represents the generalized current. This
indeed happens when the solenoid $S$ is uniquely ergodic (unique
ergodicity for a $1$-solenoid implies that all orbits are dense and
therefore minimality, by proposition 5.8 in \cite{MPM1}).

\begin{theorem} \label{thm:11.12} Let $S$ be a uniquely
ergodic oriented $1$-solenoid, and let $\cM_\cL (S)=\{\mu \}$. Let
$f:S\to M$ be an immersion. Then for each leaf $c\subset S$ we
have that $f\circ c$ is a Schwartzman asymptotic cycle with
 $$
 [f\circ c]=[f,S_\mu] \in H_1(M,\RR) \, ,
 $$
and we have
 $$
 \cC^g(f\circ c)=\cC^g ({{f,S}})=\PP\cC_{RS}({{f,S}})=\{ [f,S_\mu]\} \subset
 H_1(M,\RR ) \, .
 $$

Therefore $f:S\to M$ fully represents its Ruelle-Sullivan homology class
$[f, S_\mu]$.
%In particular, this homology class is independent of the metric $g$
%on $M$ up to a scalar factor.
\end{theorem}

%%%%%%%%%%%%%%%%%%%%%%%%%%%%%%%%%%%%%%%%%%%%%%%%%%
\section{Schwartzman $k$-dimensional cycles} \label{sec:k-schwartzman}
%%%%%%%%%%%%%%%%%%%%%%%%%%%%%%%%%%%%%%%%%%%%%%%%%%

We study in this section how to extend Schwartzman theory to
$k$-dimensional submanifolds of $M$. We assume that
$M$ is a compact $C^\infty$ Riemannian manifold.

Given an immersion
$c: N \to M$ from an oriented smooth manifold $N$ of dimension
$k\geq 1$, it is natural to consider exhaustions $(U_n)$ of $N$ with
$U_n\subset N$ being $k$-dimensional compact submanifolds with
boundary $\bd U_n$. We close $U_n$ with a $k$-dimensional oriented
manifold $\G_n$ with boundary $\bd \G_n=-\bd U_n$ (that is, $\bd
U_n$ with opposite orientation, so that $N_n=U_n\cup \G_n$ is a
$k$-dimensional compact oriented manifold without boundary), in such
a way that $c_{|U_n}$ extends to a piecewise %an immersion of
smooth map $c_n:N_n \to M$. We may consider the associated homology
class $[c_n(N_n)]\in H_k(M,\ZZ)$. By analogy with section
\ref{sec:1-schwartzman}, we consider
  \begin{equation}\label{eqn:cn(Nn)}
  \frac{1}{t_n} [c_n(N_n) ] \in H_k(M,\RR ) \, ,
  \end{equation}
for increasing sequences $(t_n)$, $t_n>0$, and $t_n \to +\infty$,
and look for sufficient conditions for (\ref{eqn:cn(Nn)}) to have
limits in $H_k(M,\RR )$. Lemma \ref{lem:9.1} extends to higher
dimension to show that, as long as we keep control of the $k$-volume
of $c_n(\G_n)$, the limit is independent of the closing procedure.

\begin{lemma} \label{lem:closing-k-dim}
Let $(\G_n)$ be a sequence of closed (i.e. compact without boundary)
oriented $k$-dimensional manifolds with piecewise smooth maps $c_n :
\G_n \to M$, and let $(t_n)$ be a sequence with $t_n>0$ and $t_n\to
+\infty$. If
 $$
 \lim_{n\to +\infty } \frac{{\Vol}_k (c_n(\G_n) )}{t_n}=0 \, ,
 $$
then in $H_k(M,\RR)$ we have
 $$
 \lim_{n\to +\infty} \frac{[c_n(\G_n)]}{t_n} =0 \, .
 $$
\end{lemma}

The proof follows the same lines as the proof of lemma
\ref{lem:9.1}. We define now $k$-dimensional Schwartzman asymptotic
cycles.

\begin{definition} \textbf{\em (Schwartzman asymptotic $k$-cycles and clusters)}
\label{def:schwartzman-cluster-k-dim} Let $c:N\to M$ be an immersion
from a $k$-dimensional oriented manifold $N$ into $M$. For all
increasing sequences $(t_n)$, $t_n\to +\infty$, and exhaustions
$(U_n)$ of $N$ by $k$-dimensional compact submanifolds with
boundary,
%open subsets with smooth boundary,
we consider all possible Schwartzman limits
 $$
 \lim_{n\to +\infty } \frac{[c_n(N_n)]}{t_n} \in H_k(M,\RR) \, ,
 $$
where $N_n=U_n\cup \G_n$ is a closed oriented manifold with
 \begin{equation}\label{eqn:condition}
 \frac{\Vol_k (c_n(\G_n) )}{t_n } \to 0 \, .
 \end{equation}
Each such limit is called a Schwartzman asymptotic $k$-cycle. These
limits form the Schwartzman cluster $\cC ({c,N})\subset H_k(M,\RR)$
of $N$.
\end{definition}

Observe that a Schwartzman limit does not depend on the choice of
the sequence $(\G_n)$, as long as it satisfies
(\ref{eqn:condition}). Note that this condition is independent of
the particular Riemannian metric chosen for $M$.

As in dimension $1$ we have

\begin{proposition}\label{prop:schwartzman-cluster-k-dim}
The Schwartzman cluster $\cC (c,N)$ is a closed cone of
$H_k(M,\RR)$.
\end{proposition}

The Riemannian structure on $M$ induces a Riemannian structure on
$N$ by pulling back by $c$. We define the Riemannian exhaustions
$(U_n)$ of $N$ as exhaustions of the form
  $$
  U_n=\bar B(x_0 , R_n) \, ,
  $$
i.e. the $U_n$ are Riemannian (closed) balls in $N$ centered at a
base point $x_0\in N$ and $R_n\to +\infty$. If the $R_n$ are
generic, then the boundary of $U_n$ is smooth

We define the Riemannian Schwartzman cluster of $N$ as follows. It
plays the role of the balanced Riemannian cluster of section
\ref{sec:1-schwartzman} for dimension $1$.

\begin{definition}\label{def:schwartzman-cluster-k-dim2}
The Riemann-Schwartzman cluster of $c:N\to M$, $\cC^g({c,N})$, is the
set of all limits, for all Riemannian exhaustions $(U_n)$,
 $$
 \lim_{n\to +\infty } \frac{1}{\Vol_k (c_n(N_n) )} [c_n(N_n)] \in
 H_k(M,\RR) \, ,
 $$
 such that $N_n=U_n\cup \G_n$ and
 \begin{equation} \label{eqn:condition2}
\frac{\Vol_k (c_n(\Gamma_n) )}{\Vol_k (c_n(N_n) )}\to 0\, .
  \end{equation}
All such limits are called Riemann-Schwartzman asymptotic
$k$-cycles.
\end{definition}

\begin{definition}\label{def:regular-asymptotic-k-cycles}
The immersed manifold $c:N\to M$ represents a homology class $a\in
H_k(M,\RR )$ if the Riemann-Schwartzman cluster $\cC^g({c,N})$
contains only $a$,
 $$
 \cC^g({c,N})=\{ a\} \, .
 $$
We denote $[{c,N}]=a$, and call it the Schwartzman homology class of $(c,N)$.
\end{definition}

Now we can define the notion of representation of homology classes
by immersed solenoids extending definition
\ref{def:representation-Schwartzman-1-solenoid} to higher dimension.

\begin{definition} \textbf{\em (Schwartzman representation of homology classes)}
\label{def:representation-Schwartzman-k-solenoid} Let $f:S\to M$ be an
immersion in $M$ of an oriented $k$-solenoid $S$. Then $S$ is a
Riemannian solenoid with the pull-back metric $f^* g$.
\begin{enumerate}
\item If $S$ is endowed with a transversal measure
 $\mu=(\mu_T)\in \cM_\cT (S)$, the immersed solenoid $f:S_\mu\to M$
 represents a homology class $a\in H_1(M, \RR)$ if for
 $(\mu_T)$-almost all leaves $l\subset S$, we have that $(f,l)$ is a
 Riemann-Schwartzman asymptotic $k$-cycle with $[{f,l}]=a$.

\item The immersed solenoid $f:S\to M$ fully represents a homology
 class $a\in H_1(M, \RR)$ if for all leaves $l\subset S$,
 we have that $(f,l)$ is a
 Riemann-Schwartzman asymptotic $k$-cycle with $[{f,l}]=a$.
\end{enumerate}
\end{definition}

\begin{definition}\textbf{\em (Equivalent exhaustions)}
Two exhaustions $(U_n)$ and $(\hat U_n)$ are equivalent if
 $$
 \frac{\Vol_k (U_n-\hat U_n) +\Vol_k(\hat U_n -U_n )}{\Vol_k (U_n)}\to 0 \, .
 $$
\end{definition}

Note that if two exhaustions $(U_n)$ and $(\hat U_n)$ are equivalent, then
 $$
 \frac{\Vol_k (\hat U_n)}{\Vol_k (U_n)}\to 1 \, .
 $$
Moreover, if $N_n=U_n\cup \Gamma_n$ are closings satisfying
(\ref{eqn:condition2}), then we may close $\hat U_n$ as follows:
after slightly modifying $\hat U_n$ so that $U_n$ and $\hat U_n$
have boundaries intersecting transversally, we glue $F_1=U_n-\hat
U_n$ to $\hat U_n$ along $F_1\cap \bd \hat U_n$, then we glue a copy
of $F_2=\hat U_n- U_n$ (with reversed orientation) to $\hat U_n$
along $F_2\cap \bd \hat U_n$. The boundary of $\hat U_n\cup F_1\cup
F_2$ is homeomorphic to $\bd U_n$, so we may glue $\G_n$ to it, to
get $\hat N_n=\hat U_n\cup F_1\cup F_2 \cup \Gamma_n$. Note that
 $$
 \Vol_k(\hat N_n) =\Vol_k(N_n) + 2\Vol_k (\hat U_n- U_n) \approx \Vol_k(N_n)\,.
 $$
Define $\hat{c}_n$ by $\hat{c}_{n|F_1}= c_{|(U_n-\hat U_n)}$,
$\hat{c}_{n|F_2}= c_{|(\hat U_n- U_n)}$ and $\hat{c}_{n|\G_n}=
c_{n|\G_n}$. Then
 $$
 [c_n(N_n)] = [\hat{c}_n(\hat N_n)]\,,
 $$
so both exhaustions define the same Schwartzman asymptotic
$k$-cycles.

\begin{definition}\textbf{\em (Controlled solenoid)} \label{def:controlled}
Let $V\subset S$ be an open subset of a solenoid $S$.
We say that $S$ is controlled by $V$ if
%there is a %finite union of flow-box $V\subset S$, such that
for any Riemann exhaustion $(U_n)$ of any leaf of $S$ there is an
equivalent exhaustion $(\hat U_n)$ such that for all $n$ we have
$\partial \hat U_n \subset V$. %We say that the solenoid is
%controlled by $V$.
\end{definition}

\begin{definition}\textbf{\em (Trapping region)} \label{def:trapping}
An open subset $W\subset S$ of a solenoid $S$ is a trapping region
if there exists a continuous map $\pi : S \to \TT$ such that
\begin{enumerate}
\item[(1)] For some $0<\epsilon_0<1/2$, $W=\pi^{-1} ((-\epsilon_0 ,\epsilon_0))$.

\medskip

\item[(2)]  There is a global transversal $T\subset \pi^{-1} (\{ 0\} )$.

\medskip

\item[(3)] Each connected component of $\pi^{-1} (\{ 0\} )$ intersects $T$ in exactly
one point.

\medskip

\item[(4)] $0$ is a regular value for $\pi$, that is, $\pi$ is smooth
in a neighborhood of $\pi^{-1} (\{ 0\} )$ and it $d\pi$ is surjective
at each point of $\pi^{-1} (\{ 0\} )$ (the differential $d\pi$ is
understood leaf-wise).

%For each open interval $]a,b[\subset \TT$, each connected
%component of $\pi^{-1} ( ]a,b[)$ is a manifold.

\medskip

\item[(5)] For each connected component $L$ of $\pi^{-1} (\TT-\{0\} )$ we
have ${\overline L}\cap T=\{ x,y\}$, where $\{x\} \in {\overline L}\cap T\cap
\pi^{-1} ((-\epsilon_0 ,0])$ and $\{y\} \in {\overline L}\cap T\cap \pi^{-1}
([0,\epsilon_0 ))$. We define $R_T : T\to T$ by $R_T(x)=y$.
\end{enumerate}

\end{definition}

Let $C_x$ be the (unique) component of $\pi^{-1} (\{ 0\} )$ through
$x\in T$. By (4), $C_x$ is a smooth $(k-1)$-dimensional manifold. By
(5), there is no holonomy in $\pi^{-1} ((-\epsilon_0 ,\epsilon_0
))$, so $C_x$ is a compact submanifold. Let $L_x$ be the connected
component of $\pi^{-1} (\TT-\{0\} )$ with ${\overline L_x}\cap T=\{
x,y\}$. This is a compact manifold with boundary
 \begin{equation}\label{eqn:Lx}
 \bd \overline L_x=C_x\cup C_y =C_x\cup C_{R_T(x)} \, .
  \end{equation}

\begin{proposition}\label{prop:trapping}
If $S$ has a trapping region $W$ with global transversal $T$, then
holonomy group of $T$ is generated by the map $R_T$.
\end{proposition}

\begin{proof}
If $\g$ is a path with endpoints in $T$, we may homotop it so that
each time it traverses $\pi^{-1} (\{ 0\} )$, it does it through $T$.
Then we may split $\g$ into sub-paths such that each path has
endpoints in $T$ and no other points in $\pi^{-1} (\{ 0\} )$. Each
of this sub-paths therefore lies in some $\overline L_x$ and has
holonomy $R_T$, $R_T^{-1}$ or the identity. The result follows.
% Let $\g\subset S$ with endpoints $x,y\in T$. We have that $\pi (\g
% )$ is a closed path in $\TT$ and $[\pi(\g )]\in H_1(\TT ,\ZZ)\approx
% \ZZ$. Since $\g$ is homothopic to $\g_1\cup\g_2\cup \ldots \cup
% \g_n$ with $[\pi (\g_i )] =1$, with $\g_i$ with endpoints
% $x_i,x_{i+1} \in \TT$, $x_0=x$ and $x_{n+1}=y$, we have that
% $$
% R_\g=R_T^{[\pi (\g )]} \, .
% $$
\end{proof}

\begin{theorem}\label{thm:trapping}
A solenoid $S$ with a trapping region $W$ is controlled by $W$.
\end{theorem}

\begin{proof}
Fix a base point $y_0\in S$ and a exhaustion $(U_n)$ of the leaf $l$
through $y_0$ of the form $U_n=\bar B(y_0, R_n)$, $R_n\to +\infty$.
Consider $x_0\in T$ so that $y_0\in \overline
L_{x_0}$. The leaf $l$ is the infinite union
 $$
 l=\bigcup_{n\in \ZZ} \overline L_{R_T^n(x_0)} \, .
 $$
If $R_T^n(x_0)=x_0$ for some $n\geq 1$ then $l$ is a compact manifold.
Then for some $N$, we have $U_N=l$, so the controlled condition of definition
\ref{def:controlled} is satisfied for $l$.

Assume that $R_T(x_0) \not= x_0$. Then $l$ is a non-compact manifold.
For integers $a < b$, denote
  \begin{equation}\label{eqn:Uab}
  \hat{U}_{a,b}:= \bigcup_{k=a}^{b-1} \overline L_{R_T^k(x_0)} \, .
  \end{equation}
This is a manifold with boundary
 $$
 \bd \hat{U}_{a,b}= C_{R_T^a(x_0)}\cup C_{R_T^b(x_0)}\, .
 $$

Given $U_n$, pick the maximum $b\geq 1$ and minimum $a\leq 0$ such
that $\hat{U}_{a,b}\subset U_n$, and denote $\hat U_n=\hat{U}_{a,b}$
for such $a$ and $b$. Clearly $\bd \hat U_n \subset W$. Let us see
that $(U_n)$ and $(\hat U_n)$ are equivalent exhaustions, i.e. that
 $$
 \frac{\Vol_k(U_n-\hat U_n)}{\Vol_k(U_n)} \to 0\,.
 $$

Let $b'\geq 1$ the minimum and $a'\leq 0$ the maximum such that
$U_n\subset \hat{U}_{a',b'}$. Let us prove that
 $$
 \Vol_k(\hat{U}_{a',b'}-\hat{U}_{a,b})
 $$
is bounded. This clearly implies the result.

Take $y\in \overline L_{R^{b'-1}_T(x_0)} \cap U_n$. Then $d(y_0,y)\leq R_n$.
By compactness of $T$, there is a lower bound $c_0>0$ for the distance from
$C_x$ to $C_{R_T(x)}$ in $L_x$, for all $x\in T$.
Taking the geodesic path from $y_0$ to $y$, we see that there are points
in $y_i\in \overline L_{R^{b'-i}_T(x_0)}$ with $d(y_0,y_i)\leq R_n - (i-2) \, c_0$,
for $2\leq i\leq b'$.

As $\overline L_{R^{b}_T(x_0)}$ is not totally contained in $U_n$, we may take
$z\in \overline L_{R^{b}_T(x_0)} -U_n$, so $d(y_0,z)>R_n$.
Both $z$ and $y_{b'-b}$ are on the same leaf $\overline L_{R^{b}_T(x_0)}'$.
By compactness of $T$, the diameter for a leaf $\overline L_x$ is bounded above by
some $c_1>0$, for all $x\in T$. So
 $$
 R_n-(b'-b-2) \, c_0 \geq d(y_0,y_{b'-b}) \geq d(y_0,z) - d(y_{b'-b},z) >R_n - c_1\, ,
 $$
hence
 $$
 b'-b < \frac{c_1}{c_0} +2\, .
 $$
Analogously,
 $$
 a-a' < \frac{c_1}{c_0} +2\, .
 $$

Again by compactness of $T$, the $k$-volumes of $\overline L_x$ are
uniformly bounded by some $c_2>0$, for all $x\in T$. So
 $$
 \Vol_k(\hat{U}_{a',b'}-\hat{U}_{a,b}) \leq (b'-b + a-a') c_2 <
 2\left(\frac{c_1}{c_0} +2\right)\, c_2 \,,
 $$
concluding the proof.
\end{proof}

\begin{theorem} \label{thm:11.12bis}
Let $S$ be a minimal oriented $k$-solenoid endowed with a transversal ergodic
measure $\mu\in\cM_\cL(S)$ and with a trapping region $W\subset S$.
Consider an immersion $f:S\to M$ such that $f(W)$ is contained in a
contractible ball in $M$. Then $f:S_\mu \to M$ represents its
Ruelle-Sullivan homology class $[f,S_\mu ]$, i.e. for $\mu_T$-almost all
leaves $l\subset S$,
 $$
 [{f,l}]=[f,S_\mu] \in H_k(M,\RR) \, .
 $$

If $S_\mu$ is uniquely ergodic, then $f:S_\mu\to M$ fully represents
its Ruelle-Sullivan homology class.

In particular, this homology class is independent of the metric $g$
on $M$ up to a scalar factor.
\end{theorem}

\begin{proof}
We define a map $\varphi_T : T\to H_k(M,\ZZ)$ as follows: given
$x\in T$, consider $f(\overline L_x)$. Since $\bd f(\overline L_x)$
is contained in a contractible ball $B$ of $M$, we can close
$f(L_x)$ locally as $N_x =f(\overline L_x)\cup \G_x$ and define an
homology class $\varphi_T (x)=[N_x]\in H_k(M,\ZZ )$. This is
independent of the choice of the closing. This map $\varphi_T$ is
measurable and bounded in $H_k(M,\ZZ)$ since the $k$-volume of
$\G_x$ may be chosen uniformly bounded.  Also we can define a map
$l_T: T\to \RR_+$ by $l_T(x)=\Vol_k(\overline L_x)$. It is also a
measurable and bounded map.

We have seen that every Riemann exhaustion $(U_n)$ is equivalent to
an exhaustion $(\hat U_n)$ with $\partial \hat U_n \subset W$. Note
also that we can saturate the exhaustion $(\hat U_n)$ into $(\hat
U_{n,m})_{n\leq 0\leq m}$, with $\hat U_{n,m}$ defined in
(\ref{eqn:Uab}), where $\partial \hat U_{n,m} = C_{R_T^n(x_0)}\cup
C_{R_T^m(x_0)}$, and $x_0\in T$ is a base point. Since $f(W)$ is
contained in a contractible ball $B$ of $M$, we can always close
$f(\hat U_{n,m})$, with a closing inside $B$, to get $N_{n,m}$
defining an homology class $[N_{n,m}]\in H_k(M,\ZZ )$. Moreover we
have
$$
[N_{n,m}]=\sum_{i=n}^{m-1} \varphi_T(R_T^i(x_0)) \, .
$$
Thus by ergodicity of $\mu$ and Birkhoff's ergodic theorem, we have that for
$\mu_T$-almost all $x_0\in T$,
$$
\frac{1}{m-n} [N_{n,m}]\to \int_T \varphi_T \ d\mu_T \, .
$$

Also
$$
\Vol_k (\hat U_{n,m})=\sum_{i=n}^{m-1} l_T(R_T^i(x_0)) \, ,
$$
where $\Vol_k(N_{n,m})$ differs from $\Vol_k (\hat U_{n,m})$ by a
bounded quantity due to the closings.
By Birkhoff's ergodic theorem, for $\mu_T$-almost all $x_0\in
T$,
$$
\frac{1}{m-n} \Vol_k f(\hat U_{n,m}) \to \int_T l_T \ d\mu_T = \mu(S)=1\, .
$$
Thus we conclude that for $\mu_T$-almost $x_0\in T$,
$$
\frac{1}{\Vol_k (N_{n,m})} [N_{n,m}] \to \int_T \varphi_T \ d\mu_T \, ,
$$
It is easy to see as in theorem \ref{thm:Ruelle-1-solenoid}
that $\int_T \varphi_T \ d\mu_T$ is the Rulle-Sullivan homology class
$[f,S_\mu]$.
\end{proof}

Actually, when $f:S\to M$ is an immersed oriented uniquely ergodic $k$-solenoid
with a trapping region which is mapped to a contractible ball in $M$,
we may prove that $f:S_\mu\to M$ fully represents the Ruelle-Sullivan homology class $[f,S_\mu]$
by checking that the exhaustion $\hat{U}_n$ satisfies the
controlled growth condition (see definition \ref{def:controlled-growth})
and using corollary \ref{cor:volume} which guarantees that
the normalized measures $\mu_n$ supported on $\hat{U}_n$
converge to the unique Schwartzman limit $\mu$.

\setcounter{section}{0}

\renewcommand{\thesection}{{Appendix}}
%%%%%%%%%%%%%%%%%%%%%%%%%%%%%%%%%%%%%%%%%%%%%%%%%%%%%%%%%%%%%%%%%%%%%
\section{Norm on the homology}\renewcommand{\thesection}{\Alph{section}}\label{sec:appendix1-norm}
%%%%%%%%%%%%%%%%%%%%%%%%%%%%%%%%%%%%%%%%%%%%%%%%%%%%%%%%%%%%%%%%%%%%%
%\renewcommand{\thesection}{\textsc{\Alph{section}}}

Let $M$ be a compact $C^\infty$ Riemannian manifold. For each
$a\in H_1(M,\ZZ)$ we define
 $$
 l(a)=\inf_{[\g ]=a} l(\g ) \, ,
 $$
where $\g$ runs over all closed loops in $M$ with homology class
$a$ and $l(\g)$ is the length of $\g$,
 $$
 l(\g) =\int_{\g} \ ds_g \, .
 $$

By application of Ascoli-Arzela it is classical to get

\begin{proposition} \label{prop:A.1}
For each $a\in H_1(M,\ZZ)$ there exists a minimizing geodesic loop
$\g$ with $[\g]=a$ such that
 $$
  l(\g)=l(a) \, .
 $$
\end{proposition}

Note that the minimizing property implies the geodesic character
of the loop. We also have

\begin{proposition} \label{prop:A.2}
There exists a universal constant $C_0=C_0(M)>0$ only depending on
$M$, such that for $a,b \in H_1(M,\ZZ)$ and $n\in \ZZ$, we have
 $$
 l(n\cdot a)\leq |n| \ l(a) \, ,
 $$
and
 $$
 l(a+b)\leq l(a)+l(b)+C_0 \, .
 $$
(We can take for $C_0$ twice the diameter of $M$.)
\end{proposition}

\begin{proof}
Given a loop $\g$, the loop $n\g$ obtained from $\g$ running
through it $n$ times (in the direction compatible the sign of $n$)
satisfies
 $$
 [n\g]=n \ [\g],
 $$
and
 $$
 l(n\g)=|n|\, l(\g)\, .
 $$
Therefore
 $$
 l(n\cdot a)\leq l(n\g)=|n|\, l(\g) \, ,
 $$
and we get the first inequality taking the infimum over $\g$.

Let $C_0$ be twice the diameter of $M$. Any two points of $M$ can
be joined by an arc of length smaller than or equal to $C_0/2$. Given
two loops $\a$ and $\b$ with $[\a ]=a$ and $[\b ]=b$, we can
construct a loop $\g$ with $[\g ]=a+b$ by picking a point in $\a$
and another point in $\b$ and joining them by a minimizing arc
which pastes together $\a$ and $\b$ running through it back and
forth. This new loop satisfies
 $$
 l(\g )=l(\a)+l(\b )+C_0 \, ,
 $$
therefore
 $$
 l(a+b)\leq l(\a)+l(\b )+C_0 \, .
 $$
and the second inequality follows.
\end{proof}

\begin{remark} \label{rem:A.3}
It is not true that $l(n\cdot a)=n\, l(\g )$ if $l(a)=l(\g)$. To
see this take a surface $M$ of genus $g\geq 2$ and two elements
$e_1,e_2\in H_1(M,\ZZ)$ such that
 $$
 l(e_1)+l(e_2)<l(e_1+e_2)\, .
 $$
(For instance we can take $M$ to be the connected sum of a large
sphere with two small $2$-tori at antipodal points, and let $e_1$,
$e_2$ be simple closed curves, non-trivial in homology, inside
each of the two tori.) Let $a=e_1+e_2$. Then
 $$
 l(n \cdot a)=l(n \cdot (e_1+e_2))\leq n\, l(e_1)+n\, l(e_2)+ C_0 \, ,
 $$
we get for $n$ large
 $$
 l(n\cdot a)<n\, l(a) \, .
 $$
\end{remark}

\begin{theorem}\textbf{\em (Norm in homology)} \label{thm:A.4}
Let $a\in H_1(M,\ZZ)$. The limit
 $$
 ||a||=\lim_{n\to +\infty } \frac{l(n\cdot a) }{n} \ \, ,
 $$
exists and is finite. It satisfies the properties
\begin{enumerate}
 \item[(i)] For $a\in H_1(M,\ZZ)$, we have $||a||=0$ if and only if
 $a$ is torsion.
 \item[(ii)] For $a\in H_1(M,\ZZ)$ and $n\in \ZZ$, we have
 $||n\cdot a||=|n|\, ||a|| $ .
 \item[(iii)] For $a,b\in H_1(M,\ZZ)$, we have
 $$
 ||a+b||\leq ||a||+||b|| \, .
 $$
\item[(iv)] $||a||\leq l(a)$.
\end{enumerate}
\end{theorem}

\begin{proof}
Let $u_n=l(n\cdot a)+C_0$. By the properties proved before, the
sequence $(u_n)$ is sub-additive
 $$
 u_{n+m}\leq u_n+u_m \, ,
 $$
therefore
 $$
 \limsup_{n\to +\infty } \frac{u_n}{n}=\liminf_{n\to +\infty }
 \frac{u_n}{n} \, .
 $$
Moreover, we have also
 $$
 \frac{u_n}{n} \leq l(a) <+\infty \, ,
 $$
thus the limit exists and is finite. Property (iv) holds.

Property (ii) follows from
 $$
 ||n\cdot a||= \lim_{m\to \infty} \frac{l(mn\cdot a)}{m} =
 |n|\, \lim_{m\to \infty} \frac{l(m|n|\cdot a)}{m|n|} = |n|\
 ||a||\, .
 $$
Property (iii) follows from
 $$
 l(n\cdot (a+b))\leq l(n\cdot a)+l(n\cdot b) +C_0\leq n\, l(a)+n\, l(b) +C_0 \, ,
 $$
dividing by $n$ and passing to the limit.

Let us check property (i). If $a$ is torsion then $n\cdot a=0$, so
$||a||=\frac1n ||n\cdot a||=0$. If $a$ is not torsion, then there
exists a smooth map $\phi:M\to S^1$ which corresponds to an
element $[\phi]\in H^1(M,\ZZ)$ with $m=\la [\phi],a\ra > 0$. Then
for any loop $\gamma:[0,1]\to M$ representing $n\cdot a$, $n>0$,
we take $\phi\circ\gamma$ and lift it to a map
$\tilde{\g}:[0,1]\to \RR$. Thus
 $$
 \tilde{\g}(1)-\tilde{\g}(0)=\la [\phi],n\cdot a\ra = m\, n\, .
 $$
Now let $C$ be an upper bound for $|d\phi|$. Then
 $$
 m\, n=|\tilde{\g}(1)-\tilde{\g}(0)|=l(\phi\circ \gamma)\leq C\,
 l(\gamma)\, ,
 $$
so $l(\g)\geq m\, n/C$, hence $l(n\cdot a)\geq m\, n/C$ and
$||a||\geq m/C$.
\end{proof}

Now we can define a norm in $H_1(M,\QQ)=\QQ \otimes H_1(M,\ZZ)$ by
 $$
 ||\lambda \otimes a ||=|\lambda|\cdot ||a|| \, ,
 $$
and extend it by continuity to $H_1(M,\RR)=\RR\otimes H_1(M,\ZZ)$.

\end{document}